%% file: ramif-gps-20.tex
\documentclass[oneside,english]{amsart}
\usepackage[T1]{fontenc}
\usepackage[utf8x]{inputenc}
\usepackage{verbatim}
\usepackage{amstext}
\usepackage{amsthm}
\usepackage{amssymb}
\usepackage[all]{xy}

\makeatletter
\numberwithin{equation}{section}
\numberwithin{figure}{section}
\theoremstyle{plain}
\newtheorem{thm}{\protect\theoremname}[section]
\theoremstyle{definition}
\newtheorem{defn}[thm]{\protect\definitionname}
\theoremstyle{remark}
\newtheorem{rem}[thm]{\protect\remarkname}
\theoremstyle{plain}
\newtheorem{lem}[thm]{\protect\lemmaname}
\theoremstyle{plain}
\newtheorem{cor}[thm]{\protect\corollaryname}
\theoremstyle{remark}
\newtheorem{notation}[thm]{\protect\notationname}
\theoremstyle{plain}
\newtheorem{prop}[thm]{\protect\propositionname}

\input{packages_and_functions.for_preamble}
\renewcommand{\to}{\arr}
\usepackage{yhmath} % wider \widehat

\usepackage[colorinlistoftodos]{todonotes}

\usepackage{bm}
\renewcommand{\bigsqcup}{\coprod}
\renewcommand{\simeq}{\cong}

\newcommand\restr[2]{{% we make the whole thing an ordinary symbol
  \left.\kern-\nulldelimiterspace % automatically resize the bar with \right
  #1 % the function
  \vphantom{\big|} % pretend it's a little taller at normal size
  \right|_{#2} % this is the delimiter
  }}

\makeatother

\usepackage{babel}
\providecommand{\corollaryname}{Corollary}
\providecommand{\definitionname}{Definition}
\providecommand{\lemmaname}{Lemma}
\providecommand{\notationname}{Notation}
\providecommand{\propositionname}{Proposition}
\providecommand{\remarkname}{Remark}
\providecommand{\theoremname}{Theorem}

\begin{document}
\title{Moduli of formal torsors II}
\author{Fabio Tonini and Takehiko Yasuda}
\address{Universitá degli Studi di Firenze, Dipartimento di Matematica e Informatica
’Ulisse Dini’, Viale Giovanni Battista Morgagni, 67/A, 50134 Firenze,
Italy}
\email{fabio.tonini@unifi.it}
\address{Department of Mathematics, Graduate School of Science, Osaka University,
Toyonaka, Osaka 560-0043, JAPAN}
\email{takehikoyasuda@math.sci.osaka-u.ac.jp}
\thanks{The first author was supported by GNSAGA of INdAM. The second author
was supported by JSPS KAKENHI Grant Numbers JP15K17510, JP18H01112
and JP18K18710.}
\begin{abstract}
Applying the authors' preceding work, we construct a version of the
moduli space of $G$-torsors over the formal punctured disk for a
finite group $G$. To do so, we introduce two Grothendieck topologies,
the sur (surjective) and luin (locally universally injective) topologies,
and define P-schemes using them as variants of schemes. Our moduli
space is defined as a P-scheme approximating the relevant moduli functor.
We then prove that Fröhlich's module resolvent gives a locally constructible
function on this moduli space, which implies that motivic integrals
appearing in the wild McKay correspondence are well-defined.
\end{abstract}

\maketitle
\global\long\def\AA{\mathbb{A}}%
\global\long\def\PP{\mathbb{P}}%
\global\long\def\NN{\mathbb{N}}%
\global\long\def\GG{\mathbb{G}}%
\global\long\def\ZZ{\mathbb{Z}}%
\global\long\def\QQ{\mathbb{Q}}%
\global\long\def\CC{\mathbb{C}}%
\global\long\def\FF{\mathbb{F}}%
\global\long\def\LL{\mathbb{L}}%
\global\long\def\RR{\mathbb{R}}%
\global\long\def\MM{\mathbb{M}}%
\global\long\def\SS{\mathbb{S}}%

\global\long\def\bx{\boldsymbol{x}}%
\global\long\def\by{\boldsymbol{y}}%
\global\long\def\bf{\mathbf{f}}%
\global\long\def\ba{\mathbf{a}}%
\global\long\def\bs{\mathbf{s}}%
\global\long\def\bt{\mathbf{t}}%
\global\long\def\bw{\mathbf{w}}%
\global\long\def\bb{\mathbf{b}}%
\global\long\def\bv{\mathbf{v}}%
\global\long\def\bp{\mathbf{p}}%
\global\long\def\bq{\mathbf{q}}%
\global\long\def\bj{\mathbf{j}}%
\global\long\def\bM{\mathbf{M}}%
\global\long\def\bd{\mathbf{d}}%
\global\long\def\bA{\mathbf{A}}%
\global\long\def\bB{\mathbf{B}}%
\global\long\def\bC{\mathbf{C}}%
\global\long\def\bP{\mathbf{P}}%
\global\long\def\bX{\mathbf{X}}%
\global\long\def\bY{\mathbf{Y}}%
\global\long\def\bZ{\mathbf{Z}}%
\global\long\def\bW{\mathbf{W}}%
\global\long\def\bV{\mathbf{V}}%
\global\long\def\bU{\mathbf{U}}%
\global\long\def\bN{\mathbf{N}}%
\global\long\def\bQ{\mathbf{Q}}%

\global\long\def\cN{\mathcal{N}}%
\global\long\def\cW{\mathcal{W}}%
\global\long\def\cY{\mathcal{Y}}%
\global\long\def\cM{\mathcal{M}}%
\global\long\def\cF{\mathcal{F}}%
\global\long\def\cX{\mathcal{X}}%
\global\long\def\cE{\mathcal{E}}%
\global\long\def\cJ{\mathcal{J}}%
\global\long\def\cO{\mathcal{O}}%
\global\long\def\cD{\mathcal{D}}%
\global\long\def\cZ{\mathcal{Z}}%
\global\long\def\cR{\mathcal{R}}%
\global\long\def\cC{\mathcal{C}}%
\global\long\def\cL{\mathcal{L}}%
\global\long\def\cV{\mathcal{V}}%
\global\long\def\cU{\mathcal{U}}%
\global\long\def\cS{\mathcal{S}}%
\global\long\def\cT{\mathcal{T}}%
\global\long\def\cA{\mathcal{A}}%
\global\long\def\cB{\mathcal{B}}%
\global\long\def\cG{\mathcal{G}}%
\global\long\def\cP{\mathcal{P}}%
\global\long\def\cQ{\mathcal{Q}}%

\global\long\def\fs{\mathfrak{s}}%
\global\long\def\fp{\mathfrak{p}}%
\global\long\def\fm{\mathfrak{m}}%
\global\long\def\fX{\mathfrak{X}}%
\global\long\def\fV{\mathfrak{V}}%
\global\long\def\fx{\mathfrak{x}}%
\global\long\def\fv{\mathfrak{v}}%
\global\long\def\fY{\mathfrak{Y}}%
\global\long\def\fa{\mathfrak{a}}%
\global\long\def\fb{\mathfrak{b}}%
\global\long\def\fc{\mathfrak{c}}%
\global\long\def\fO{\mathcal{O}}%
\global\long\def\fd{\mathfrak{d}}%
\global\long\def\fP{\mathfrak{P}}%

\global\long\def\rv{\mathbf{\mathrm{v}}}%
\global\long\def\rx{\mathrm{x}}%
\global\long\def\rw{\mathrm{w}}%
\global\long\def\ry{\mathrm{y}}%
\global\long\def\rz{\mathrm{z}}%
\global\long\def\bv{\mathbf{v}}%
\global\long\def\bw{\mathbf{w}}%
\global\long\def\sv{\mathsf{v}}%
\global\long\def\sx{\mathsf{x}}%
\global\long\def\sw{\mathsf{w}}%

\global\long\def\Spec{\mathrm{Spec}\,}%
\global\long\def\Hom{\mathrm{Hom}}%

\global\long\def\Var{\mathbf{Var}}%
\global\long\def\Gal{\mathrm{Gal}}%
\global\long\def\Jac{\mathrm{Jac}}%
\global\long\def\Ker{\mathrm{Ker}}%
\global\long\def\Image{\mathrm{Im}}%
\global\long\def\Aut{\mathrm{Aut}}%
\global\long\def\st{\mathrm{st}}%
\global\long\def\diag{\mathrm{diag}}%
\global\long\def\characteristic{\mathrm{char}}%
\global\long\def\tors{\mathrm{tors}}%
\global\long\def\sing{\mathrm{sing}}%
\global\long\def\red{\mathrm{red}}%
\global\long\def\ord{\mathrm{ord}}%
\global\long\def\pt{\mathrm{pt}}%
\global\long\def\op{\mathrm{op}}%
\global\long\def\Val{\mathrm{Val}}%
\global\long\def\Res{\mathrm{Res}}%
\global\long\def\Pic{\mathrm{Pic}}%
\global\long\def\disc{\mathrm{disc}}%
\global\long\def\Coker{\mathrm{Coker}}%
 
\global\long\def\length{\mathrm{length}}%
\global\long\def\sm{\mathrm{sm}}%
\global\long\def\rank{\mathrm{rank}}%
\global\long\def\age{\mathrm{age}}%
\global\long\def\et{\mathrm{et}}%
\global\long\def\hom{\mathrm{hom}}%
\global\long\def\tor{\mathrm{tor}}%
\global\long\def\reg{\mathrm{reg}}%
\global\long\def\cont{\mathrm{cont}}%
\global\long\def\crep{\mathrm{crep}}%
\global\long\def\Stab{\mathrm{Stab}}%
\global\long\def\discrep{\mathrm{discrep}}%
\global\long\def\mld{\mathrm{mld}}%
\global\long\def\GCov{G\textrm{-}\mathrm{Cov}}%
\global\long\def\P{\mathrm{P}}%

\global\long\def\GL{\mathrm{GL}}%
\global\long\def\codim{\mathrm{codim}}%
\global\long\def\prim{\mathrm{prim}}%
\global\long\def\cHom{\mathcal{H}om}%
\global\long\def\cSpec{\mathcal{S}pec}%
\global\long\def\Proj{\mathrm{Proj}\,}%
\global\long\def\modified{\mathrm{mod}}%
\global\long\def\ind{\mathrm{ind}}%
\global\long\def\rad{\mathrm{rad}}%
\global\long\def\Conj{\mathrm{Conj}}%
\global\long\def\fie{\textrm{-}\mathrm{fie}}%
\global\long\def\NS{\mathrm{NS}}%
\global\long\def\Disc{\mathrm{Disc}}%
\global\long\def\hotimes{\hat{\otimes}}%
\global\long\def\Fil{\mathrm{Fil}}%
\global\long\def\Inn{\mathrm{Inn}}%
\global\long\def\rfil{\mathrm{rfil}}%
\global\long\def\per{\mathrm{per}}%
\global\long\def\id{\mathrm{id}}%
\global\long\def\AffVar{\mathbf{AffVar}}%
\global\long\def\Alg{\mathbf{Alg}}%
\global\long\def\PSch{\textit{P-}\mathbf{Sch}}%
\global\long\def\PQVar{\textit{P-}\mathbf{QVar}}%
\global\long\def\QVar{\mathbf{QVar}}%
\global\long\def\Set{\mathbf{Set}}%
\global\long\def\Eis{\mathrm{Eis}}%
\global\long\def\ACF{\mathbf{ACF}}%
\global\long\def\SDelta{\mathcal{S}\Delta}%
\global\long\def\adm{\mathrm{adm}}%
\global\long\def\Aff{\mathbf{Aff}}%
\global\long\def\Affnat{\mathbf{Aff}^{\natural}}%
\global\long\def\hatAffnat{(\mathbf{Aff}^{\natural})^{\land}}%
\global\long\def\Algnat{\mathbf{Alg}^{\natural}}%
\global\long\def\Sch{\mathbf{Sch}}%
\global\long\def\Shsur{\mathrm{Sh_{sur}}}%
\global\long\def\Shubi{\mathrm{Sh_{ubi}}}%
\global\long\def\Schnat{\mathbf{Sch}^{\natural}}%
\global\long\def\unif{\mathrm{unif}}%
\global\long\def\sets{\mathbf{Set}}%
\global\long\def\bDelta{\pmb{\bm{\Delta}}}%
\global\long\def\uni{\mathrm{uni}}%
\global\long\def\Uni{\mathrm{Unif}}%
\global\long\def\sep{\mathrm{sep}}%
\global\long\def\Unit{\mathrm{Unit}}%

\tableofcontents{}

\input{packages_and_functions.tex}

\makeatletter 
\providecommand\@dotsep{5} 
\makeatother 
%\listoftodos\relax

\section{Introduction}

In the preceding paper \cite{Tonini:2017qr}, the authors constructed
the moduli stack of $G$-torsors over $\Spec k((t))$, where $k$
is a field of characteristic $p>0$ and $G$ is a group of the form
$H\rtimes C$ for a $p$-group $H$ and a tame cyclic group $C$,
which generalizes and refines Harbater's work for $p$-groups \cite{MR579791}.
The motivation of the authors came from the wild McKay correspondence.
In this theory, motivic integrals of the forms $\int_{\Delta_{G}}\LL^{d-v}$
and $\int_{\Delta_{G}}\LL^{w}$ appear, where $\Delta_{G}$ is the
moduli space of $G$-torsors over $\Spec k((t))$, $v,w$ are functions
$\Delta_{G}\longrightarrow\frac{1}{|G|}\ZZ$ associated to a representation
$G\longrightarrow\GL_{d}(k[[t]])$ and $d$ is its rank. The first
aim of the present paper is to construct a version of the moduli space
$\Delta_{G}$ for an arbitrary finite group by using the mentioned
result from the previous paper and prove that motivic integrals as
above are well-defined in a version of the complete Grothendieck ring
of varieties. After the first draft of this paper had been written,
the main assertion of the wild McKay correspondence, the equality
of $\int_{\Delta_{G}}\LL^{d-v}$ and the stringy motive of the quotient
$k[[t]]$-scheme $\AA_{k[[t]]}^{d}/G$ associated to the representation,
was proved by the second author in \cite{yasuda2019motivic}, building
on the well-definedness of the integral obtained in this paper.

We do not construct the moduli stack, since it appears difficult.
Instead we construct what we call the P-moduli space. This is a version
of the moduli space, which is even coarser than the coarse moduli
space. Actually this is the coarsest one for which motivic integrals
as above still make sense. We construct the category of P-schemes
by modifying morphisms of the category of schemes. The P-moduli space
is the P-scheme approximating the relevant moduli functor the most.
We call it the strong P-moduli space if it satisfies an additional
condition. A precise statement of our first main result is as follows:
\begin{thm}[Theorem \ref{thm:strong P moduli of DeltaG}]
\label{thm:main1}Let $G$ be a finite group and let $k$ be a field.
Consider the functor from the category of affine $k$-schemes to the
category of sets which sends $\Spec R$ to the set of isomorphism
classes of $G$-torsors over $\Spec R((t))$. This functor has a strong
P-moduli space, which is the disjoint union of countably many affine
schemes of finite type over $k$.
\end{thm}

The theorem can be generalized to the case where $G$ is a finite
étale group scheme (Corollary \ref{cor:P-module-etale-gp-scheme})
and it holds in any characteristic. Here we outline the proof. From
the previous work, we have the P-moduli space if $G$ is the semidirect
product of a $p$-group and a tame cyclic group. We construct the
P-moduli space for an arbitrary $G$ by ``gluing'' the P-moduli
spaces of semidirect products as above. To do so, we show that every
$G$-torsor over $\Spec R((t))$ is induced from an $H$-torsor with
$H\subset G$ a subgroup which is a semidirect product as above, locally
in $\Spec R$ for some Grothendieck topology. What we use as such
a topology is the sur (surjective) topology; a scheme morphism $Y\longrightarrow X$
is a sur covering if it is surjective and locally of finite presentation.
This topology is also incorporated into the very definition of P-schemes.
We also introduce the luin (locally universally injective) topology.
It is interesting that such a crude topology as the sur topology is
still useful. The sur and luin topologies and P-schemes would be of
independent interest and we study their basic properties. We note
that Kelly \cite[Def. 3.5.1]{kelly-thesis} introduced a Grothendieck
topology similar to the sur topology; he does not assume that a covering
$Y\to X$ is locally of finite presentation, instead assume that every
point $x\in X$ admits a lift $y\in Y$ having the same residue field
as $x$.

Advantages of P-moduli spaces are that it is much easier to show their
existence than in the case of usual moduli stacks or schemes and that
they are invariant by some transformations preserving geometric points.
For instance if $\shF$ is a moduli functor or stack and $f\colon\shF\to\frac{1}{l}\Z$
is a locally constructible function, then a P-moduli space for $\shF$
is a disjoint union of $P$-moduli spaces for $\{f=r\}\subseteq\shF$,
the submoduli of $\shF$ of objects where $f$ has constant value
$r$ (see \ref{prop:functor of constant maps}). For instance, we
may restrict ourselves to those $G$-torsors over $\Spec R((t))$
which have constant ramification as a family over $\Spec R$ in a
suitable sense.

We also prove that the functions $v,w$ mentioned above are locally
constructible. This together with Theorem \ref{thm:main1} shows that
integrals $\int_{\Delta_{G}}\LL^{d-v}$ and $\int_{\Delta_{G}}\LL^{w}$
are well-defined. The function $v$ is essentially the same as the
module resolvent introduced by Fröhlich \cite{MR0414520} and $w$
is a variant of $v$. When the given representation $G\longrightarrow\GL_{d}(k[[t]])$
is a permutation representation, then $v$ and $w$ are closely related
to the Artin and Swan conductors \cite{MR0414520,MR3431631}.

\subsection*{Acknowledgements}

We would like to thank Ofer Gabber for helpful comments, which allowed
us to remove the Noetherianity assumption in Theorem \ref{cor:stratified-uniformization}.

\section{Notation, terminology and convention}

For a scheme $X$, we denote by $|X|$ the underlying topological
space. 

For a category $\bC$, the expression $A\in\bC$ means that $A$ is
an object of $\bC$. 

We denote the category of schemes by $\Sch$ and the one of affine
schemes by $\Aff$. For a scheme $S$, we denote the category of $S$-schemes
by $\Sch/S$. When $S$ is separated, we denote by $\Aff/S$ its subcategory
of $S$-schemes affine over $S$. 

We often identify a ring $R$ with its spectrum $\Spec R$ and apply
the terminology for schemes also to rings. For instance, for a ring
map $A\longrightarrow B$ and a finite group $G$, we say that $B$
is a $G$-torsor over $A$ or that $B/A$ is a $G$-torsor if $\Spec B\longrightarrow\Spec A$
is a $G$-torsor.

\section{Luin and sur topologies\label{sec:Luin-and-sur-topologies}}

In this section, we introduce two Grothendieck topologies, the luin
topology and the sur topology, and study their basic properties. We
need these topologies to develop the theory of P-schemes and P-moduli
spaces in Section \ref{sec:P-schemes-and-moduli}.
\begin{defn}
\label{def:universally stuff} A morphism of schemes $f\colon Y\longrightarrow X$
is said to be \emph{universally bijective }(resp. \emph{universally
injective}) if for all maps of schemes $X'\arr X$ the map $X'\times_{X}Y\arr X'$
is bijective (resp. injective) as map of sets.

Let $S$ be a base scheme. A morphism of $S$-schemes $f\colon Y\longrightarrow X$
is said to be \emph{geometrically bijective }(resp. \emph{geometrically
injective}, \emph{geometrically surjective}) if for all algebraically
closed field $K$ and maps $\Spec K\to S$ the map $\Hom_{S}(\Spec(K),Y)\to\Hom_{S}(\Spec(K),X)$
is bijective (resp. injective, surjective).
\end{defn}

\begin{rem}
A morphism of $S$-schemes $f\colon Y\longrightarrow X$ is geometrically
bijective (resp. injective, surjective) if and only if it is so as
a map of ($\Spec\Z$-)schemes. In other words this notion is an absolute
property, not a relative one. We have chosen to include a base scheme
to make the definition compatible with \ref{def:geometrically stuff for functors}.
\end{rem}

\begin{lem}
\label{lem:univ-bi}Let $S$ be a base scheme and $f\colon Y\longrightarrow X$
be a morphism of $S$-schemes. We have:
\begin{enumerate}
\item the map $f$ is universally injective if and only it is geometrically
injective;
\item if the map $f$ is geometrically bijective (resp. geometrically surjective)
then it is universally bijective (resp. surjective); the converse
holds if $f$ is locally of finite type.
\end{enumerate}
\end{lem}

\begin{proof}
(1) This is \cite[Tag 01S4]{stacks-project}, taking into account
that, if $K\to L$ is a map of fields and $Z$ is a scheme, then $Z(K)\to Z(L)$
is injective.

(2) Since the base change of a surjective map is surjective, the ``bijective
case'' follows from (1) and the ``surjective'' one. If a map if
geometrically surjective then it is clearly surjective. For the converse
assume that $f$ is locally of finite type and let $x\colon\Spec K\longrightarrow X\in X(K)$
be a geometric point. The fiber $Y\times_{X}\Spec K$ is locally of
finite type, non empty since $f$ is surjective, and has a $K$-rational
point since $K$ is algebraically closed. This defines an element
of $Y(K)$ over $x\in X(K)$.
\end{proof}
\begin{rem}
\label{rem:geom bettern than univ}Even if the notion of universally
injective or bijective is more common in the literature, we are going
to use the notion of geometrically injective or bijective instead.
Firstly because, for morphisms locally of finite type, the two notions
coincide. But the main reason is that the second notion easily extends
in the case of natural transformation of functors (see \ref{def:geometrically stuff for functors})
and it plays a crucial role in the next chapters.
\end{rem}

\begin{defn}
A morphism of schemes $g\colon Y\arr X$ is called a \emph{sur (surjective)
covering} if it is locally of finite presentation and surjective.

A morphism of schemes $g\colon Y\arr X$ is called a \emph{luin (locally
universally injective) covering} if it is a sur covering and there
is a covering $\{Y_{i}\}_{i}$ of open subsets of $Y$ such that $Y_{i}\arr X$
is geometrically injective. 

A morphism of schemes $g\colon Y\arr X$ is called a \emph{ubi (universally
bijective) covering} if it is a geometrically injective sur covering.
In particular it is a geometrically bijective luin covering.

We call a collection of morphisms $(U_{i}\longrightarrow X)_{i\in I}$
in $\Sch$ a sur (resp. luin, ubi) covering if the induced morphism
$g\colon Y=\coprod_{i\in I}U_{i}\longrightarrow X$ is a sur (resp.
luin, ubi) covering. 
\end{defn}

It is easy to check that luin and sur coverings satisfy the axioms
of a Grothendieck topology. We define the \emph{luin topology} and
the \emph{sur topology} by these collections of coverings. By construction
fppf coverings are sur coverings, while Zariski coverings are luin
coverings.

If $(U_{i}\longrightarrow X)_{i\in I}$ is a sur (resp. luin) covering,
then $\coprod_{i\in I}U_{i}\arr X$ is again a sur (resp. luin) covering.
Hence, when discuss the luin or sur topology, we often consider coverings
$U\longrightarrow X$ consisting of a single morphism.

We will soon prove (see \ref{cor:finiteness for sur coverings}) that
sur coverings satisfy the equivalent conditions of \cite[Proposition 2.33]{MR2223406},
that is for a sur covering any open affine below is dominated by a
quasi-compact open above. This is the classical quasi-compactness
condition required for fpqc coverings. 
\begin{defn}[{\cite[Tag 005G]{stacks-project}}]
 A subset $E\subseteq X$ of a topological space is called \emph{constructible
}if it is a finite union of sets of the form $U\cap(X-V)$ where $U,V\arr X$
are quasi-compact open immersions. It is said \emph{locally constructible
}if there exists an open covering $\{U_{i}\}_{i}$ of $X$ such that
$E\cap U_{i}$ is constructible in $U_{i}$.
\end{defn}

Every constructible subset of a quasi-compact space is quasi-compact
(see \cite[Tag 09YH]{stacks-project}). For a quasi-compact and quasi-separated
scheme, locally constructible subsets are constructible (see \cite[Tag 054E]{stacks-project}).
We will often use this form of Chevalley's theorem (see \cite[Tag 054K]{stacks-project}):
\begin{thm}
\label{thm:Chevalley} A quasi-compact and locally of finite presentation
map of schemes preserves locally constructible subsets. 
\end{thm}

\begin{lem}
\label{lem:scheme structure finitely presented}Let $X=\Spec A$ be
an affine scheme and $U,V\subseteq X$ two quasi-compact open subsets.
Then there is a scheme structure on $E=U\cap(X-V)$ such that $E\arr X$
is a finitely presented immersion. If $U$ is affine we can furthermore
assume that $E$ is affine.
\end{lem}

\begin{proof}
The quasi-compactness of $V$ implies that there exists a finitely
generated ideal $I$ of $A$ such that $\Spec(A/I)=X-V$. The composition
$(\Spec A/I)\cap U\arr\Spec A/I\arr\Spec A$ is a finitely presented
immersion whose image is $U\cap(X-V)$ and it is affine if $U$ is
affine.
\end{proof}
\begin{lem}
\label{lem:quasi-compact constructible topology} Let $Y$ be a quasi-compact
and quasi-separated scheme and $E\subseteq Y$ a constructible subset.
If $E=\bigcup_{j\in J}E_{j}$ is union of constructible subsets of
$Y$ then there exists $J'\subseteq J$ finite such that $E=\bigcup_{j\in J'}E_{j}$.

\begin{comment}
We use the theory of spectral spaces (see \cite[Tag 08YF]{stacks-project}).
A quasi-compact and quasi-separated scheme is a spectral space. Any
spectral space endowed with the coarsest topology in which its constructible
subsets are both open and closed are quasi-compact. Moreover, by \cite[Tag 0902]{stacks-project},
all its constructible subsets with the induced topology are spectral
spaces too. Thus we just have to show that the $E_{j}$ are constructible
in the topology of $E$. Since intersection of quasi-compact open
subsets in $Y$ are quasi-compact it is enough to show that, if $U\subseteq Y$
is a quasi-compact open then $E\cap U$ is quasi-compact. This follows
from the fact that $E\cap U$ is constructible and therefore quasi-compact.
\end{comment}
\end{lem}

\begin{proof}
We use the theory of spectral spaces (see \cite[Tag 08YF]{stacks-project}).
A quasi-compact and quasi-separated scheme is a spectral space. Any
spectral space endowed with the coarsest topology in which its constructible
subsets are both open and closed is quasi-compact. With respect to
this topology, $E$ is a closed subset of $Y$, hence quasi-compact
and $E=\cup_{j\in J}E_{j}$ is an open covering. This implies the
assertion.
\end{proof}
\begin{cor}
\label{cor:reducing to finitely many}Let $\{f_{j}\colon Z_{j}\arr Y\}_{j\in J}$
be a collection of locally finitely presented and quasi-compact maps
such that the image of $\coprod_{j\in J}Z_{j}\arr Y$ is locally constructible.
If $V\subseteq Y$ is a quasi-compact and quasi-separated open subset
of $Y$ (e.g. affine) then there exists a finite subset $J'\subseteq J$
such that
\[
\Imm(\coprod_{j\in J'}f_{j}^{-1}(V)\arr V)=\Imm(\coprod_{j\in J}f_{j}^{-1}(V)\arr V).
\]
\end{cor}

\begin{proof}
If $E=\Imm(\coprod_{j\in J}Z_{j}\arr Y)$ then the right hand side
of the above equation is $E\cap V$, which is constructible in $V$.
By Chevalley's theorem \ref{thm:Chevalley} the image $E_{i}$ of
$f_{j}^{-1}(V)\arr V$ is constructible. The conclusion follows from
\ref{lem:quasi-compact constructible topology}.
\end{proof}
\begin{cor}
\label{cor:ubi is quasi-compact} Let $f\colon X\arr Y$ be a geometrically
injective map which is locally of finite presentation. Then $f$ is
quasi-compact if and only if $f(X)$ is locally constructible (e.g.
if $f$ is geometrically bijective).
\end{cor}

\begin{proof}
The ``only if'' part is Chevalley's theorem \ref{thm:Chevalley},
while the ``if'' part follows reducing first to the case where $Y$
is affine and then applying \ref{cor:reducing to finitely many} with
$\{Z_{j}\}_{j}$ an open affine covering of $X$.
\end{proof}
\begin{cor}
\label{cor:finiteness for sur coverings} Let $f\colon X\arr Y$ be
a sur covering. Then for any quasi-compact open $V$ of $Y$ there
exists a quasi-compact open $W$ of $X$ such that $f(W)=V$. In other
words sur coverings satisfy the equivalent conditions stated in \cite[Proposition 2.33]{MR2223406}.
\end{cor}

\begin{proof}
We can assume $V=Y$ affine. In this case it is enough to apply \ref{cor:reducing to finitely many}
to the collection $\{U\arr Y\}_{U}$ with $U$ open affine of $X$.
\end{proof}
\begin{defn}
An open covering $\{U_{i}\}_{i\in I}$ of a topological space $Y$
is called \emph{locally finite} if for all $y\in Y$ there exists
an open neighborhood $U_{y}$ of $y$ such that there are at most
finitely many indices $i\in I$ with $U_{y}\cap U_{i}\neq\emptyset$.
If all the $U_{i}$ are quasi-compact (e.g. affine) this is the same
of asking that for each $i\in I$ there are at most finitely many
$j\in I$ with $U_{i}\cap U_{j}\neq\emptyset$.
\end{defn}

\begin{notation}
\label{nota:splitting a covering by constructible sets} Let $Y$
be a set and let $\{Z_{i}\}_{i\in I}$ be a collection of subsets
$Z_{i}\subset Y$. For a subset $J\subset I$, we define $Z_{J}^{\circ}:=\bigcap_{i\in J}Z_{i}\setminus\bigcup_{i\in J^{c}}Z_{i}$. 
\end{notation}

It is easy to see that $Y=\bigsqcup_{J\subset I}Z_{J}^{\circ}$. Moreover,
for subsets $J_{1},J_{2}\subset I$, the set $Z_{J_{1},J_{2}}:=\bigcap_{i\in J_{1}}Z_{i}\setminus\bigcup_{i\in J_{2}}Z_{i}$
is written as 
\begin{equation}
Z_{J_{1},J_{2}}=\bigsqcup_{J_{1}\subset J,\,J_{2}\subset J^{c}}Z_{J}^{\circ}.\label{eq:disjoint union}
\end{equation}

\begin{lem}
\label{lem:key lemma for scheme structure on constructible sets}
Let $Y$ be a scheme and $E\subseteq Y$ be a locally constructible
subset. Then there exist affine schemes $Z_{j}$ and locally finitely
presented immersions $Z_{j}\arr Y$ such that the map $\coprod_{j}Z_{j}\arr Y$
has image $E$. 

If $Y$ is quasi-separated and has a locally finite and affine open
covering we can furthermore assume that the maps $Z_{j}\arr Y$ are
quasi-compact and the map $\bigsqcup_{j}Z_{j}\arr Y$ is a finitely
presented monomorphism. In particular, $Y$ has a ubi covering $\{Z_{j}\longrightarrow Y\}$
with $Z_{j}$ affine.
\end{lem}

\begin{proof}
We first prove the second assertion. Let $Y$ be a quasi-separated
scheme with a locally finite affine open covering $\{U_{i}\}_{i\in I}$.
Consider the decomposition $Y=\bigsqcup_{J\subset I}U_{J}^{\circ}$
as sets. From the local finiteness of the covering, if $J$ is infinite,
then $U_{J}^{\circ}$ is empty. Since $\{U_{i}\}$ is a covering,
if $J$ is empty, then so is $U_{J}^{\circ}$. If $J$ is finite and
non empty and $j\in J$ then $U=\cap_{q\in J}U_{q}$ is a quasi-compact
open subset of $U_{j}$ because $Y$ is quasi-separated. Since the
covering $\{U_{i}\}_{i}$ is locally finite there are only finitely
many $q\in J^{c}$ such that $U_{q}\cap U_{j}\neq\emptyset$ and therefore,
using again that $Y$ is quasi-separated, the union
\[
V=\bigcup_{q\in J^{c}}(U_{q}\cap U_{j})
\]
is a quasi-compact open subset of $U_{j}$. Since $U_{J}^{\circ}=U\cap(U_{j}-V)$
by definition, Lemma \ref{lem:scheme structure finitely presented}
yields a structure of scheme on $U_{J}^{\circ}$ such that the morphism
$U_{J}^{\circ}\longrightarrow U_{j}$ is a finitely presented immersion
and, if $Y$ is also separated so that $U$ is affine, we can choose
$U_{J}^{\circ}$ affine. In general $U_{J}^{\circ}$ is quasi-compact
and separated and, since $Y$ is quasi-separated, also the map $U_{J}^{\circ}\to Y$
is a finitely presented immersion. Indeed the map $U_{j}\to Y$ is
a finitely presented immersion: it is locally of finite presentation
because an open immersion, quasi-compact because $Y$ is quasi-separated
and $U_{j}$ is quasi compact, and quasi-compact because it is a monomorphism.
Moreover the map $\coprod_{J}U_{J}^{\circ}\longrightarrow Y$ is a
surjective monomorphism.

We use the above construction several times. Firstly, starting from
any locally finite affine open covering $\{U_{i}\}_{i\in I}$, it
allows to reduce the problem to the case where $Y$ is quasi-compact
and separated: we can replace $Y$ with $U_{J}^{\circ}$ and $E$
with its preimage on $U_{J}^{\circ}$.

In this case, since $E$ is constructible, we can write $E=\bigcup_{l=1}^{n}V_{l}\setminus V_{n+l}$
with quasi-compact open subsets $V_{l}\subset Y$. We now apply the
above constructions to a finite affine open covering $\{U_{i}\}_{i\in I}$
such that all $V_{l}$ can be written as union of some opens in this
covering. Since each $V_{l}\setminus V_{n+l}$ is a (automatically
disjoint) union of subsets of the form $U_{J}^{\circ}$, $J\subset I$
as in (\ref{eq:disjoint union}), so is $E$, say $E=\bigsqcup_{J\in\Lambda}U_{J}^{\circ}$
for a set $\Lambda$ of subsets of $I$. It follows that the map $\coprod_{J\in\Lambda}U_{J}^{\circ}\longrightarrow Y$
is finitely presented, a monomorphism, has image $E$ and all $U_{J}^{\circ}\to Y$
are finitely presented immersions. Moreover, since $Y$ is separated,
the $U_{J}^{\circ}$ are affine, as required.

For a general scheme $Y$, we take an affine covering $\{Y_{i}\}_{i}$
of $Y$ and a finitely presented monomorphism $\coprod_{j}Z_{ij}\arr Y_{i}$
with image $E\cap Y_{i}$ and such that $Z_{ij}\arr Y_{i}$ is an
affine and finitely presented immersion. It is clear that $\coprod_{i,j}Z_{ij}\arr Y$
satisfies the requests.
\end{proof}
\begin{cor}
\label{cor:best ubi covering} Let $Y$ be a quasi-separated scheme
with a locally finite and affine open covering. Then if $f\colon Z\arr Y$
is a luin covering there exists an ubi covering $Z'\arr Y$ which
is finitely presented, separated and has a factorization $Z'\arr Z\arrdi fY$.
In particular a luin covering of $Y$ is refined by an ubi covering. 
\end{cor}

\begin{comment}
By \ref{lem:key lemma for scheme structure on constructible sets}
with $E=Y$ we can assume that $Y$ is an affine scheme. In particular,
taking an open subset, we can also assume $Z$ quasi-compact. Now
let $\{Z_{i}\}_{i\in I}$ be a finite open covering by affine schemes
of $Z$ such that $Z_{i}\arr Y$ is geometrically injective and set
$f(Z_{i})=E_{i}$. Since $f$ is quasi-compact and locally of finite
presentation all the $E_{i}$ are constructible subsets. Given $J\subseteq I$
consider the constructible subset
\[
E_{J}=\bigcap_{j\in J}E_{j}\cap\bigcap_{j\notin J}(Y-E_{j})
\]
It is easy to see that all $E_{i}$ and $Y$ are a disjoint union
of sets of the form $E_{J}$. By \ref{lem:key lemma for scheme structure on constructible sets}
there exists a scheme structure on $E_{J}$ such that $E_{J}\arr Y$
is a finitely presented monomorphism. For all $J$ choose an index
$i_{J}\in I$ and set $Z_{J}=E_{J}\times_{Y}Z_{i}$ and the projection
$\alpha_{J}\colon Z_{J}\arr E_{J}$. This map is surjective and it
also has the same property of $Z_{i}\arr Y$: geometrically injective,
affine and of finite presentation. The scheme $Z'=\bigsqcup_{J}Z_{J}$
with the map $Z'\arr Y$ satisfy the requests.
\end{comment}

\begin{proof}
By \ref{lem:key lemma for scheme structure on constructible sets},
there exists a ubi covering $\{Y_{j}\to Y\}$ with $Y_{j}$ affine.
For each $j$, we have the luin covering $Z_{j}:=Z\times_{Y}Y_{j}\to Y_{j}$.
If $Z_{j}'\to Y_{j}$ is a ubi covering as in the corollary for this
luin covering, then $\coprod_{j}Z_{j}'\to Y$ is the desired ubi covering.
Thus it suffices to show the collary in the case where $Y$ is affine.

By \ref{cor:finiteness for sur coverings} there exists a quasi-compact
open subset $\widetilde{Z}\subseteq Z$ such that $f(\widetilde{Z})=Y$.
In particular $\widetilde{Z}\to Y$ is a luin covering refining the
given one. Thus we can assume that $Z$ is quasi-compact. Now let
$\{Z_{i}\}_{i\in I}$ be a finite open covering by affine schemes
of $Z$ such that $Z_{i}\arr Y$ is geometrically injective and set
$f(Z_{i})=E_{i}$. Since $f$ is quasi-compact and locally of finite
presentation all the $E_{i}$ are constructible subsets. We have $Y=\bigsqcup_{\emptyset\ne J\subset I}E_{J}^{\circ}$
as sets. For each $J$, we choose an index $i_{J}\in J$ and let $Z_{J}\subset Z_{i_{J}}$
to be the preimage of $E_{J}^{\circ}$, which is a constructible subset
of $Z_{i_{J}}$ mapping bijectively onto $E_{J}^{\circ}$. Again from
\ref{lem:key lemma for scheme structure on constructible sets}, there
exists a finitely presented morphism $W_{J}\longrightarrow Z_{i_{J}}$
from an affine scheme $W_{J}$ whose image is $Z_{J}$. The scheme
$Z'=\bigsqcup_{J}W_{J}$ with the map $Z'\arr Y$ satisfies the requests.
\end{proof}

\section{P-schemes and moduli spaces\label{sec:P-schemes-and-moduli}}

In this section, we develop the theory of P-schemes and P-moduli spaces.
The category of P-varieties (P-schemes of finite type) can be regarded
as the categorification of the modified Grothendieck ring of varieties
(Definition \ref{def:Grothedieck rings}). Although it would be more
natural from this viewpoint to use the luin topology to define P-schemes,
we actually use the sur topology. This is a key in later applications,
since we have uniformization only locally in the sur topology (Section
\ref{sec:Uniformization}).
\begin{notation}
\label{nota:Homs} Starting from this section, we make use of Hom
sets such as $\Hom_{S}(-,-)$, $\Hom_{\Sch/S}(-,-)$, $\Hom_{\Aff/S}(-,-)$
and we want to clarify here the notation. If $\shC$ is a category
and $X,Y\colon\shC^{\op}\to\sets$ are two functor then $\Hom_{\shC}(X,Y)$
denotes the set of natural transformations $X\to Y$. The symbols
$\Hom_{\Sch/S}(-,-)$, $\Hom_{\Aff/S}(-,-)$ are used with this meaning
for $\shC=\Sch/S,\Aff/S$ respectively. On the other hand, by abuse
of notation, we will often simply write $\Hom_{S}(-,-)$ if it is
clear which category is used: $\Hom_{S}(X,Y)$ would be $\Hom_{\Sch/S}(X,Y)$
for $X,Y\colon(\Sch/S)^{\op}\to\sets$, $\Hom_{\Aff/S}(X,Y)$ for
$X,Y\colon(\Aff/S)^{\op}\to\sets.$

If $X$ and $Y$ are $S$-schemes, then $X,Y$ can be thought of as
functors $h_{X},h_{Y}\colon(\Sch/S)^{\op}\to\sets$ but also as their
restriction $h_{X}^{a},h_{Y}^{a}\colon(\Aff/S)^{\op}\to\sets$. By
Yoneda's lemma and Zariski descent we have canonical isomorphisms
\[
\Hom_{S}(X,Y)\simeq\Hom_{\Sch/S}(h_{X},h_{Y})\simeq\Hom_{\Aff/S}(h_{X}^{a},h_{Y}^{a})
\]
 where $\Hom_{S}(X,Y)$ denotes the set of morphisms as $S$-schemes.
We will simply write $\Hom_{S}(X,Y)$ and, depending on the interpretation
of $X,Y$, use one of the above sets.
\end{notation}

\subsection{P-morphisms and associated functor}

Let $S$ be a base scheme. Let $\Sch/S$ (resp. $\Aff/S$) be the
category of $S$-schemes (resp. affine schemes over $S$). By $\Sch'/S$
we denote either $\Sch/S$ or $\Aff/S$. As is well-known, associating
the functor $T\longmapsto X(T)$ to the scheme $X$, we have a fully
faithful embedding of $\Sch/S$ into the category of functors $(\Sch'/S)^{\op}\longrightarrow\Set$.
We often identify an $S$-scheme with the associated functor $(\Sch'/S)^{\op}\longrightarrow\Set$.
\begin{defn}
\label{def:ACF and the F functor}We denote by $\ACF/S$  the category
of algebraically closed fields $K$ together with a map $\Spec K\arr S$.
Given a functor $X\colon(\Sch'/S)^{\op}\arr\sets$ (e.g. an $S$-scheme)
we denote by $X_{F}$ the restriction 
\[
X_{F}\colon\ACF/S\arr(\Sch'/S)^{\op}\arr\sets.
\]
Two elements $x\colon\Spec K\to X$ and $y\colon\Spec K'\to X$ of
$X_{F}$ are equivalent if there exists a commutative diagram   \[   \begin{tikzpicture}[xscale=2.3,yscale=-1.2]     \node (A0_0) at (0, 0) {$\Spec K''$};     \node (A0_1) at (1, 0) {$\Spec K'$};     \node (A1_0) at (0, 1) {$\Spec K$};     \node (A1_1) at (1, 1) {$X$};     \path (A0_0) edge [->]node [auto] {$\scriptstyle{}$} (A0_1);     \path (A1_0) edge [->]node [auto] {$\scriptstyle{x}$} (A1_1);     \path (A0_1) edge [->]node [auto] {$\scriptstyle{y}$} (A1_1);     \path (A0_0) edge [->]node [auto] {$\scriptstyle{}$} (A1_0);   \end{tikzpicture}   \] where
$K''$ is a field. We denote by $|X|$ the set of equivalence classes
of maps as above and we call points of $X$ its elements. If $x\in X_{F}(K)$,
with an abuse of notation, we will write $x\in|X|$. If $f\colon Y_{F}\to X_{F}$
is a map of functors $(\ACF/S)^{\op}\arr\sets$ we denote by $|f|$,
or sometimes simply by $f$, the induced map $|Y|\to|X|$.
\end{defn}

\begin{rem}
\label{rem:points of functor and subsets} Given a subset $E$ of
$|X|$ we define $E_{F}\subseteq X_{F}$ by
\[
E_{F}(K)=\{s\in X(K)\st s\in E\subseteq|X|\}.
\]
Conversely given a subfunctor $G\subseteq X_{F}$ we can define the
subset $|G|\subseteq|X|$ of points $p$ for which there exists $x\in G(K)$
such that $x=p$ in $|X|$. Clearly $|E_{F}|=E$ for all $E\subseteq X$
and, in particular $|X_{F}|=|X|$. We also have $G\subseteq|G|_{F}$
with an equality if: for all $x\in X(K)$ and $K\arr K'$ if $x_{|K'}\in G(K')$
then $x\in G(K)$.
\end{rem}

Notice that $(-)_{F}$ preserves fiber products of functors $(\Sch'/S)^{\op}\arr\sets$.
If $X,Y$ are two schemes over $S$ with Yoneda functors $h_{X/S},h_{Y/S}\colon\Sch'/S\to\sets$
then $h_{X/S}\times h_{Y/S}$ is the Yoneda functor of the $S$-scheme
$X\times_{S}Y$, that is $h_{X\times_{S}Y/S}=h_{X/S}\times h_{Y/S}\colon\Sch'/S\to\sets$.
In particular we have $(X\times_{S}Y)_{F}=X_{F}\times Y_{F}$.
\begin{defn}
\label{def:P-morphism} Given an $S$-scheme $Y$ and a functor $X\colon(\Sch'/S)^{\op}\arr\sets$,
a $P$-\emph{morphism} $Y\arr X$ (over $S$) is a natural transformation
$f\colon Y_{F}\arr X_{F}$ for which there exist a sur covering $\{g_{i}\colon Z_{i}\arr Y\}$
over $S$ and morphisms $f_{i}'\colon Z_{i}\arr X$ over $S$ making
the following diagrams commutative  \begin{equation}
 \begin{tikzpicture}[xscale=1.9,yscale=-1.2]
    \node (A0_0) at (0, 0) {$(Z_i)_F$};     \node (A1_0) at (0, 1) {$Y_F$};     \node (A1_1) at (1, 1) {$X_F$};     \path (A1_0) edge [->]node [auto] {$\scriptstyle{f}$} (A1_1);     \path (A0_0) edge [->]node [auto,swap] {$\scriptstyle{(g_i)_F}$} (A1_0);     \path (A0_0) edge [->]node [auto] {$\scriptstyle{(f_i')_F}$} (A1_1);   \end{tikzpicture} 
\end{equation} We denote by $\Hom_{S}^{\textup{P}}(Y,X)\subseteq\Hom(Y_{F},X_{F})$
the set of P-morphisms from $Y$ to $X$.

Since $P$-morphisms are stable by composition we define $\PSch/S$
as the category whose objects are $S$-schemes and whose maps are
P-morphisms over $S$. An $S$ \emph{P-scheme} means an $S$ scheme
regarded as an object of $\PSch/S$.

If $X\colon(\Sch'/S)^{\op}\arr\sets$ is a functor we define $X^{\P}\colon(\Sch'/S)^{\op}\arr\sets$
as follows: $X^{\P}(Y)=\Hom_{S}^{\P}(Y,X)\subseteq\Hom(Y_{F},X_{F})$
is the set of $P$-morphisms $Y\arr X$.
\end{defn}

There exists a natural functor $\Sch/S\longrightarrow\PSch/S$ sending
an $S$-scheme to itself and a morphism to the induced P-morphism.
Notice moreover that a $P$-morphism of schemes $Y\arr X$, more generally
a functor $Y_{F}\arr X_{F}$, induces a map on the sets of points
$|Y|\arr|X|$ which in general is not continuous.
\begin{rem}
We coined the terms, P-morphism and P-scheme, to connote ``perfect''
and piecewise. Indeed relative or absolute Frobenius maps of varieties
in positive characterstic and, more generally, ubi coverings (for
example given by locally closed decompositions) become isomorphisms
as P-morphisms (see \ref{lem:when phiP is an isomorphism}). In particular
their inverses are examples of P-morphisms which are not necessarily
morphisms of schemes. In the case of decompositions those morphisms
do not even define continuous maps on the underlying topological spaces.
\end{rem}

\begin{prop}
\label{prop:remarks on P and F} Let $X\colon(\Sch'/S)^{\op}\arr\sets$
be a functor.
\begin{enumerate}
\item The functor $X^{\P}\colon(\Sch'/S)^{\op}\arr\sets$ extends naturally
to a functor $(\PSch/S)^{\op}\arr\sets$.
\item There is a canonical morphism $X\arr X^{\P}$ and $X_{F}\arr(X^{\P})_{F},|X|\to|X^{\P}|$
are isomorphisms. Moreover $X^{\P}\arr(X^{\P})^{\P}$ is an isomorphism.
\item Let $Y\in\Sch'/S$, $f\colon Y\arr X$ a P-morphism and $\overline{f}\colon Y\arr X^{\P}$
the corresponding element. Then $Y_{F}\arrdi fX_{F}\simeq(X^{\P})_{F}$
coincides with $\overline{f}_{F}\colon Y_{F}\arr(X^{\P})_{F}$. 
\item If $Y\colon(\Sch'/S)^{\op}\arr\sets$ is another functor and using
that $Y_{F}\simeq(Y^{\P})_{F}$ we obtain a map 
\[
\Hom_{S}(X,Y^{\P})\arr\Hom_{S}(X_{F},Y_{F})
\]
and this map is injective. If $X$ is an $S$-scheme its image is
$\Hom_{S}^{\P}(X,Y)$.
\item \label{enu:fully-faithful}If $Y\colon(\Sch'/S)^{\op}\arr\sets$ is
another functor then a map $X\arr Y^{\P}$ factors uniquely through
a map $X^{\P}\arr Y^{\P}$. In other words the map $X\arr X^{\P}$
induces a bijection 
\[
\Hom_{S}(X^{\P},Y^{\P})\arr\Hom_{S}(X,Y^{\P}).
\]
In particular if $X$ is an $S$-scheme then
\[
\Hom_{S}(X^{\P},Y^{\P})\simeq\Hom_{S}(X,Y^{\P})\simeq\Hom_{S}^{\P}(X,Y)\subseteq\Hom_{S}(X_{F},Y_{F}).
\]
\item If $Y\colon\Sch/S\arr\sets$ is a sheaf in the Zariski topology and
$X$ is an $S$-scheme then
\[
\Hom_{S}^{\P}(X,Y)=\Hom_{S}^{P}(X,Y_{|\Aff/S})\subseteq\Hom_{S}(X_{F},Y_{F}).
\]
In particular the restriction 
\[
\Hom_{S}(X,Y^{\P})\arr\Hom_{S}(X_{|\Aff/S},(Y_{|\Aff/S})^{\P})
\]
 is an isomorphism.
\item If $U$ is a reduced $S$-scheme and $X$ is a scheme the map $X(U)\arr X^{\P}(U)$
is injective . 
\end{enumerate}
\end{prop}

\begin{proof}
$1)$ Consider the functor $\overline{X}\colon\PSch/S\to\sets$ given
by $\overline{X}(Y)=\Hom_{\ACF/S}(Y_{F},X_{F})$. The extension of
$X^{\P}$ from $\Sch'/S$ to $\PSch/S$ is the subfunctor of $\overline{X}$
given by $Y\longmapsto\Hom_{S}^{\P}(Y,X)$.

$2)$ If $T\in\ACF/S$ then the composite map
\[
X(T)=X_{F}(T)\to X^{\P}(T)=(X^{\P})_{F}(T)\hookrightarrow\Hom_{\ACF/S}(T{}_{F},X_{F})
\]
is a bijection from the Yoneda lemma. This shows that the left map
is a bijection and that $X_{F}\arr(X^{\P})_{F}$ is an isomorphism.

The last statement follows easily from the definitions. 

$3)$ By definition of P-morphism and $X^{\P}$ we can replace $Y$
by a sur covering and assume that $f\colon Y_{F}\arr X_{F}$ is induced
by a map $\hat{f}\colon Y\arr X$. In this case $Y\arrdi{\hat{f}}X\arr X^{\P}$
is exactly $\overline{f}$ and taking $(-)_{F}$ the conclusion follows. 

$4)$ In the first claim one can replace $X$ by a scheme in $\Sch'/S$,
in which case the result follows easily from $3)$. Assume now that
$X$ is an $S$-scheme and consider the last claim. If $X\in\Sch'/S$
the result is again clear by $3)$. If $X$ is not necessarily in
$\Sch'/S$, then $\Hom_{S}(X,Y^{\P})$ can be identified with the
set of transformations $f\colon X_{F}\to Y_{F}$ such that, for all
$u\colon T\to X$ with $T\in\Sch'/S$, the composition $f\circ u_{F}\colon T_{F}\to Y_{F}$
is a P-morphism. Thus $\Hom_{S}^{\P}(X,Y)\subseteq\Hom_{S}(X,Y^{\P})$.
Conversely if $f\in\Hom_{S}(X,Y^{\P})$, then it is Zariski locally
a P-morphism and therefore sur locally induced by a genuine morphism.
But this exactly means that $f$ is a P-morphism, as required.

$5)$ Given a map $\phi\colon X\arr Y^{P}$ it is easy to see that
the unique extension $\phi^{\P}\colon X^{\P}\arr Y^{\P}$ is defined
as follows. Given $a\colon U\arr X^{\P}$ one defines 
\[
\phi^{\P}(a)\colon U_{F}\arrdi{a_{F}}(X^{\P})_{F}\simeq X_{F}\arrdi{\phi}(Y^{\P})_{F}\simeq Y_{F}.
\]

$6)$ It is enough to note that from the sheaf condition on $Y$ it
follows that the map 
\[
\Hom_{\Sch/S}(Z,Y)\arr\Hom_{\Aff/S}(Z_{|\Aff/S},Y_{|\Aff/S})
\]
is an isomorphism for all $S$-schemes $Z$.

$7)$ Consider $a,b\colon U\arr X$ such that $a_{F}=b_{F}$ and the
Cartesian diagram   \[   \begin{tikzpicture}[xscale=2.1,yscale=-1.2]     \node (A0_0) at (0, 0) {$W$};     \node (A0_1) at (1, 0) {$U$};     \node (A1_0) at (0, 1) {$X$};     \node (A1_1) at (1, 1) {$X\times_S X$};     \path (A0_0) edge [->]node [auto] {$\scriptstyle{h}$} (A0_1);     \path (A1_0) edge [->]node [auto] {$\scriptstyle{\Delta}$} (A1_1);     \path (A0_1) edge [->]node [auto] {$\scriptstyle{(a,b)}$} (A1_1);     \path (A0_0) edge [->]node [auto] {$\scriptstyle{}$} (A1_0);   \end{tikzpicture}   \] Since
$(-)_{F}$ preserves fiber products, it follows that $h_{F}\colon W_{F}\to U_{F}$
is an isomorphism, that is $h\colon W\to U$ is geometrically bijective.
In particular $h(W)=U$ is closed as a subset of $U$. As $h$ is
an immersion it follows that it is a closed immersion (see \cite[Tag 01IQ]{stacks-project})
and therefore also an homeomorphism. Since $U$ is reduced the map
$h\colon W\to U$ is an isomorphism, which means $a=b$.
\end{proof}
\begin{lem}
\label{lem:P-morphisms to finitely presented stuff} Let $Y$ be a
scheme over $S$ and $X$ a scheme locally of finite presentation
and quasi-separated over $S$. Given a natural transformation $f\colon Y_{F}\arr X_{F}$
the following are equivalent:
\begin{enumerate}
\item there exist a luin covering $g\colon Z\arr Y$ over $S$ and a map
$f'\colon Z\arr X$ over $S$ such that the diagram \begin{equation} \label{P-morphism-2}
 \begin{tikzpicture}[xscale=1.9,yscale=-1.2]
    \node (A0_0) at (0, 0) {$Z_F$};     \node (A1_0) at (0, 1) {$Y_F$};     \node (A1_1) at (1, 1) {$X_F$};     \path (A1_0) edge [->]node [auto] {$\scriptstyle{f}$} (A1_1);     \path (A0_0) edge [->]node [auto,swap] {$\scriptstyle{g_F}$} (A1_0);     \path (A0_0) edge [->]node [auto] {$\scriptstyle{f'_F}$} (A1_1);   \end{tikzpicture} 
\end{equation}  is commutative;
\item the transformation $f$ is a $P$-morphism;
\item if $\Gamma_{f}\colon Y_{F}\arr Y_{F}\times X_{F}=(Y\times_{S}X)_{F}$
is the graph of $f$ then $C=\Imm(|\Gamma_{f}|)\subseteq Y\times_{S}X$
is locally constructible. 
\end{enumerate}
\end{lem}

\begin{proof}
$(1)\Rightarrow(2)$ Follows by definition.

$(2)\Rightarrow(3)$ This is a local statement so that we can assume
that $Y$ is affine. Let $Z\arrdi gY$ be a sur covering and $Z\arrdi{f'}X$
such that the diagram (\ref{P-morphism-2}) is commutative. By \ref{cor:finiteness for sur coverings}
we can furthermore assume that $Z\arr Y$ is quasi-compact. The set
$C=\Imm(|\Gamma_{f}|)$ is the image of the graph $\Gamma_{f'}$ of
$f'$ along $Z\times_{S}X\arr Y\times_{S}X$. By Chevalley's theorem
\ref{thm:Chevalley} it is therefore enough to show that $\Gamma_{f'}$
is locally constructible in $Z\times_{S}X$. But $X\arr S$ and therefore
$Z\times_{S}X\arr Z$ are locally of finite presentation and quasi-separated.
Therefore a section $Z\arr Z\times_{S}X$ is quasi-compact and locally
of finite presentation. Chevalley's theorem \ref{thm:Chevalley} again
shows that $\Gamma_{f'}$ is locally constructible.

$(3)\Rightarrow(1)$ By definition of luin coverings the statement
is local in $Y$, so that we can assume $Y$ affine. By \ref{lem:key lemma for scheme structure on constructible sets}
there are locally finitely presented immersions $Z_{j}\arr Y\times_{S}X$
from affine schemes such that $g\colon Z=\bigsqcup_{j}Z_{j}\arr Y\times_{S}X$
has image $C$. The composition $u\colon Z\arr Y\times_{S}X\arr Y$
is locally of finite presentation because $X\longrightarrow S$ and
hence $Y\times_{S}X\to Y$ are locally of finite presentation. It
is surjective because $C\arr Y$ is surjective. By definition it follows
that $u\colon Z\to Y$ is a sur covering. We are going to show that
it is a luin covering, more precisely that $Z_{j}\to Y$ is geometrically
injective, and that $Z_{F}\arrdi{u_{F}}Y_{F}\arrdi fX_{f}$ is induced
by the morphism $Z\to Y\times_{S}X\to X$, which will end the proof.

We first show that $C_{F}=\Imm\Gamma_{f}\subseteq Y_{F}\times X_{F}$
(see \ref{rem:points of functor and subsets}). The inclusion $\Imm\Gamma_{f}\subseteq C_{F}$
follows by construction. For the converse, let $K\in\ACF/S$ and $(y,x)\in C_{F}(K)\subseteq Y_{F}(K)\times X_{F}(K)$.
Since $C=\Imm|\Gamma_{f}|$, there exists $K\to K'$ such that $(y,x)_{|K'}=\Gamma_{f}(\overline{y})=(\overline{y},f(\overline{y}))$
for some $\overline{y}\in Y_{F}(K')$. Thus $\overline{y}=y_{|K'}$
and $x_{|K'}=f(y_{|K'})=f(y)_{|K'}$, which implies that $x=f(y)$
because $X(K)\arr X(K')$ is injective. In other words $\Gamma_{f}(y)=(y,x)\in\Imm\Gamma_{f}$.

Since $\Gamma_{f}\colon Y_{F}\to C_{F}$ is an isomorphism and the
map $g_{F}\colon Z_{F}\to(Y\times_{S}X)_{F}=Y_{F}\times X_{F}$ has
image insides $C_{F}$, there is a factorization $g_{F}\colon Z_{F}\arrdi{\phi}Y_{F}\arrdi{\Gamma_{f}}Y_{F}\times X_{F}$.
Moreover there is a commutative diagram   \[   \begin{tikzpicture}[xscale=2.1,yscale=-1.2, on top/.style={preaction={draw=white,-,line width=#1}}, on top/.default=6pt]    
\node (A1_0) at (0, 1) {$Z_F$};     
\node (A1_1) at (1, 1) {$Y_F$};     
\node (A1_2) at (2, 1) {$Y_F\times X_F$};     
\node (A1_3) at (3, 1) {$X_F$};     
\node (A2_2) at (2, 2) {$Y_F$};     
\path (A1_0) edge [->,bend left=15]node [auto,swap] {$\scriptstyle{u_F}$} (A2_2);        
\path (A1_0) edge [->]node [auto] {$\scriptstyle{\phi}$} (A1_1);     
\path (A1_1) edge [->]node [auto] {$\scriptstyle{\Gamma_f}$} (A1_2);     
\path (A1_1) edge [->]node [auto] {$\scriptstyle{\id}$} (A2_2);     
\path (A1_1) edge [->,bend right=60]node [auto] {$\scriptstyle{f}$} (A1_3);     
\path (A1_2) edge [->]node [auto] {$\scriptstyle{}$} (A1_3);     
\path (A1_2) edge [->]node [auto] {$\scriptstyle{}$} (A2_2);   
\path (A1_0) edge [->,bend right=60, on top]node [auto] {$\scriptstyle{g_F}$} (A1_2);  
\end{tikzpicture}   \] It follows that $\phi=u_{F}\colon Z_{F}\to Y_{F}$ and that $f\circ\phi\colon Z_{F}\to X_{F}$
is induced by $Z\to Y\times_{S}X\to X$. Since $Z_{j}\to Y\times_{S}X$
is geometrically injective and $g_{F}=\Gamma_{f}\circ u_{F}\colon Z_{F}\to(Y\times_{S}X)_{F}$
we can conclude that $Z_{j}\to Z\arrdi uY$ is geometrically injective,
as required.
\end{proof}
\begin{defn}
\label{def:geometrically stuff for functors} A natural transformation
$f\colon P\to Q$ of functors $(\ACF/S)^{\op}\to\sets$ is said to
be \emph{geometrically bijective }(resp. \emph{geometrically injective},
\emph{geometrically surjective}) if it is an isomorphism (resp. injective,
surjective).

A morphism $f\colon Y\longrightarrow X$ of functors $(\Sch'/S)^{\op}\longrightarrow\Set$
is said to be \emph{geometrically bijective }(resp. \emph{geometrically
injective}, \emph{geometrically surjective}) if so is $f_{F}\colon Y_{F}\to X_{F}$.
Similarly a P-morphism $g\colon Z\to X$ from an $S$-scheme is said
to be \emph{geometrically bijective }(resp. \emph{geometrically injective},
\emph{geometrically surjective}) if $g$, thought of as a map $Z_{F}\to X_{F}$,
is so.
\end{defn}

\begin{rem}
The above definition extends the one given for morphisms of schemes
(see \ref{def:universally stuff}).

If a morphism $f\colon Y\longrightarrow X$ of functors $(\Sch'/S)^{\op}\longrightarrow\Set$
is geometrically bijective (resp. geometrically injective, geometrically
surjective) then $|f|\colon|Y|\to|X|$ is bijective (resp. injective,
surjective), but the converse is not true.
\end{rem}

\begin{lem}
\label{lem:criterion one for Piso} Let $f\colon Y_{F}\longrightarrow X_{F}$
be a P-morphism of $S$-schemes. If $f$ is an isomorphism in $\PSch/S$
then it is an isomorphism as natural transformation. The converse
holds if $Y$ and $X$ are locally finitely presented and quasi-separated
over $S$.
\end{lem}

\begin{proof}
The first statement is clear. For the second, by \ref{lem:P-morphisms to finitely presented stuff}
we have that $C=\Imm(|\Gamma_{f}|)$ is locally constructible in $Y\times_{S}X$.
Since $\Imm(|\Gamma_{f^{-1}}|)$ is the image of $C$ via the automorphism
$Y\times_{S}X\simeq X\times_{S}Y$, it follows that $\Imm(|\Gamma_{f^{-1}}|)$
is locally constructible and therefore $f^{-1}\colon X_{F}\arr Y_{F}$
is a P-morphism by \ref{lem:P-morphisms to finitely presented stuff}.
\end{proof}
\begin{lem}
\label{lem:sheaves}Let $X\colon(\Sch'/S)^{\op}\arr\sets$ be a functor.
\begin{enumerate}
\item The functor $X^{\P}\colon(\Sch'/S)^{\op}\arr\sets$ is a sheaf in
the sur topology.
\item If the diagonal of $X$ is representable and of finite presentation
then $X\arr X^{P}$ is a sheafification morphism for the sur topology. 
\item If $X$ is a scheme locally of finite presentation and quasi-separated
over $S$ then $X\arr X^{\P}$ is a sheafification morphism for the
luin topology.
\end{enumerate}
\end{lem}

\begin{proof}
$1)$ Consider $\overline{X}\colon(\Sch'/S)^{\op}\arr\sets$ defined
by $\overline{X}(Y)=\Hom(Y_{F},X_{F})$. Since $(-)_{F}$ preserves
fiber products and, if $V\arr U$ is a sur covering, the maps $V_{F}(K)\arr U_{F}(K)$
are surjective for all algebraically closed fields, it is easy to
see that $\overline{X}$ is a sheaf in the sur topology. Moreover
$X^{\P}\subseteq\overline{X}$ is a subfunctor. By definition of $X^{\P}$
we see that if an object of $\overline{X}$ is sur locally in $X^{\P}$
then it belongs to $X^{\P}$. This implies that $X^{\P}$ is a subsheaf
of $\overline{X}$.

$2)$ The map $X\arr X^{\P}$ is by definition an epimorphism in the
sur topology. We need to check that if $a,b\in X(U)$ become equal
in $X^{\P}(U)$ then they are sur locally equal. In particular we
can assume $U$ affine. Let $W\arr U$ be the base change of the diagonal
$X\arr X\times_{S}X$ along the map $(a,b)\colon U\arr X\times_{S}X$.
By hypothesis $W$ is an algebraic space and $W\arr U$ is of finite
presentation. Moreover $a,b$ become equal in $X(W)$. Since $a_{F}=b_{F}\colon U_{F}\arr X_{F}$
we can conclude that $W_{F}\arr U_{F}$ is bijective. Since $W$ is
an algebraic space there exists an étale atlas $V\arr W$ from a scheme.
The resulting map $V\arr U$ is locally of finite presentation, surjective
and therefore a sur covering.

$3)$ Now assume that $X$ is a scheme locally of finite presentation
and quasi-separated over $S$. By \ref{lem:P-morphisms to finitely presented stuff}
it follows that $X\arr X^{\P}$ is an epimorphism in the luin topology.
As before we need to check that if $a,b\in X(U)$ become equal in
$X^{\P}(U)$ then they are luin locally equal. By the same argument
above we see that they are equal after a map $W\arr U$ which is locally
of finite presentation, quasi-compact and geometrically bijective.
The difference now is that $W$ is a scheme and therefore $W\arr U$
is a luin covering.
\end{proof}
\begin{cor}
\label{cor:Pschemes into sheaves} Let $\Shsur(\Sch'/S)$ be the category
of sur sheaves on $\Sch'/S$. The functor $(-)^{\P}$ determines a
fully faithful functor 
\[
\PSch/S\arr\Shsur(\Sch'/S)
\]
\end{cor}

\begin{proof}
For an $S$-scheme $X$, the functor $X^{P}$ is a sur sheaf from
Lemma \ref{lem:sheaves}. From Proposition \ref{prop:remarks on P and F}
(\ref{enu:fully-faithful}), for $S$-schemes $X$ and $Y$, we have
a natural bijection 
\[
\Hom_{S}(X^{P},Y^{P})\longrightarrow\Hom_{S}^{P}(X,Y),
\]
which proves the corollary. 
\end{proof}

\subsection{P-moduli spaces}
\begin{defn}
Let $\cF\colon(\Sch'/S)^{\op}\arr\sets$ be a functor. A \emph{P-moduli
space }of $\cF$ is an $S$-scheme $X$ together with a morphism $\pi\colon\shF\arr X^{\P}$
such that\emph{}

\begin{enumerate}
\item $\pi$ is geometrically bijective, that is the induced map $\pi_{F}\colon\shF_{F}\arr(X^{\P})_{F}$
is an isomorphism;
\item $\pi$ is universal, that is for any morphism $g:\cF\longrightarrow Y^{\P}$
where $Y$ is a scheme over $S$, there exists a unique $S$-morphism
$f:X^{\P}\longrightarrow Y^{\P}$ with $f\circ\pi=g$.
\end{enumerate}
If this is the case, we also call the morphism $\pi$ a \emph{P-moduli
space}. It is clear that if exists, a $P$-moduli space is unique
up to unique P-isomorphism.
\end{defn}

\begin{lem}
\label{lem:when phiP is an isomorphism} Let $\phi\colon\shF\arr\shG$
be a map of functors $(\Sch'/S)^{\op}\arr\sets$. If $\phi$ is geometrically
bijective then $\phi^{\P}\colon\shF^{\P}\arr\shG^{\P}$ is a monomorphism.
If $\phi$ is also an epimorphism in the sur topology then $\phi^{\P}\colon\shF^{\P}\arr\shG^{\P}$
is an isomorphism.

In particular a locally finitely presented and geometrically bijective
morphism of $S$-schemes, that is an ubi covering, is a $P$-isomorphism.
\end{lem}

\begin{proof}
The first claim follows from \ref{prop:remarks on P and F}, $4)$
and $5)$. For the second one it is enough to recall that $\shF^{\P}$
and $\shG^{\P}$ are sheaves in the sur topology.
\end{proof}
\begin{rem}
\label{rem:remarks of P-moduli spaces} Let $\pi\colon\shF\arr X^{\P}$
be a geometrically bijective map, so that, by \ref{lem:when phiP is an isomorphism},
$\pi^{\P}\colon\shF^{\P}\arr X^{\P}$ is a monomorphism. Then $\pi$
is a P-moduli space if and only if for all maps $\shF\arr Y^{\P}$
the map $(X^{\P})_{F}\simeq\shF_{F}\arr(Y^{\P})_{F}$ is a P-morphism
$X\arr Y$.

In particular $\shF\arr X^{\P}$ is a P-moduli space if and only if
$\shF^{\P}\arr X^{\P}$ is a $P$-moduli space and vice versa.
\end{rem}

\begin{defn}
Let $\cF\colon(\Sch'/S)^{\op}\arr\sets$ be a functor. A \emph{strong}
\emph{P-moduli space }for $\cF$ is an $S$-scheme $X$ together with
a morphism $\pi\colon\shF\arr X^{\P}$ such that $\pi^{\P}\colon\shF^{\P}\arr X^{\P}$
is an isomorphism.
\end{defn}

A strong $P$-moduli space is also unique up to unique P-isomorphism.
From Remark \ref{rem:remarks of P-moduli spaces} it follows that
a strong P-moduli space is a P-moduli space.
\begin{prop}
\label{prop:admissible-moduli} Let $\cF\colon(\Sch'/S)^{\op}\arr\sets$
be a functor, $X$ an $S$-scheme and $\pi\colon\shF\arr X^{\P}$
be a morphism. Then $\pi$ is an epimorphism in the sur topology if
and only if there exist a sur covering $\{Z_{i}\arrdi{g_{i}}X\}$
and commutative diagrams    \[   \begin{tikzpicture}[xscale=1.5,yscale=-1.2]     \node (A0_0) at (0, 0) {$Z_i$};     \node (A0_1) at (1, 0) {$X$};     \node (A1_0) at (0, 1) {$\shF$};     \node (A1_1) at (1, 1) {$X^\P$};     \path (A0_0) edge [->]node [auto] {$\scriptstyle{g_i}$} (A0_1);     \path (A1_0) edge [->]node [auto] {$\scriptstyle{\pi}$} (A1_1);     \path (A0_1) edge [->]node [auto] {$\scriptstyle{}$} (A1_1);     \path (A0_0) edge [->]node [auto] {$\scriptstyle{}$} (A1_0);   \end{tikzpicture}   \]The
map $\pi\colon\shF\arr X^{\P}$ is a strong P-moduli space if and
only if it is geometrically bijective and an epimorphism in the sur
topology.
\end{prop}

\begin{proof}
The first statement follows from the fact that $X\arr X^{\P}$ is
an epimorphism in the sur topology. The second one from \ref{lem:when phiP is an isomorphism}.
\end{proof}
\begin{rem}
\label{rem:strong moduli of algebraic spaces} From the point of view
of moduli theory a more natural definition of P-moduli space would
have been to admits $S$-algebraic spaces in the above definitions.
Since a quasi-separated algebraic space has a dense open subset which
is a scheme, it follows that for a finite dimensional quasi-separated
algebraic space $Y$ there exists a geometrically bijective map $X_{1}\text{\ensuremath{\amalg}}\cdots\amalg X_{n}\arr Y$
in which all $X_{i}$ are schemes and all maps $X_{i}\arr Y$ are
immersions. In particular such a $Y$ always has a strong P-moduli
space.

So in concrete cases there is no need to use algebraic spaces and
also this let us avoid to deal with locally constructible subsets
of algebraic spaces.
\end{rem}

\begin{defn}
\label{def:geometric and constructible property} By a \emph{geometric
property }$\shQ$ for a functor\emph{ $\cF\colon(\Sch'/S)^{\op}\arr\sets$
}we mean a subset $\shQ\subseteq|\shF|$. By \ref{rem:points of functor and subsets}
this can be thought of as a subfunctor $\shQ$ of $\shF_{F}$ with
the following property: for all maps $a\colon K\arr K'$ in $\ACF/S$
and all $x\in\shF(K)$, if $x$ is mapped by $\shF(K)\arr\shF(K')$
to an element of $\shQ(K')$ then $x\in\shQ(K)$. Given a geometric
property $\shQ$ of $\shF$ we define the subpresheaf 
\[
\cF^{\shQ}(V)=\{V\arrdi A\shF\mid A(V)\subseteq\shQ\subseteq|\shF|\}.
\]
A \emph{locally constructible }property for $\shF$ is a geometric
property $\shQ$ for $\shF$ satisfying the following condition: for
every $S$-scheme $V$ and map $A\colon V\to\shF$ the inverse image
$A^{-1}(\shQ)\subseteq V$ is a locally constructible subset of $V$.
\end{defn}

\begin{rem}
If $X$ is an $S$-scheme a geometric property (resp. locally constructible
property) of $X$ is a subset (resp. a locally constructible subset)
of $X$.
\end{rem}

\begin{prop}
\label{prop:locally constructible subsets as subfunctors}Let $X$
be an $S$-scheme and $\shQ$ be a locally constructible subset of
$X$. Let also $Q=\bigsqcup_{i}Q_{i}\arr X$ be a geometrically injective
map with image $\shQ$ and where the maps $Q_{i}\arr X$ are finitely
presented immersions. Then the map $Q\arr X^{\shQ}$ induces an isomorphism
$Q^{\P}\simeq(X^{\shQ})^{\P}$.
\end{prop}

\begin{proof}
Since the image of $Q\arr X$ is in $\shQ$ the map $Q\arr X^{\shQ}$
is geometrically bijective. Moreover if $V$ is a scheme with a map
$V\arr X^{\shQ}$, that is a map $V\arr X$ with image in $\shQ$,
then $Q\times_{X^{\shQ}}V=Q\times_{X}V\arr V$ is geometrically bijective,
locally of finite presentation and therefore, by \ref{cor:ubi is quasi-compact},
a luin covering. By \ref{lem:when phiP is an isomorphism} we get
the result.
\end{proof}
\begin{rem}
\label{rem:taking the geometric subfunctor commutes with taking P}
If $\shF\colon(\Sch'/S)^{\op}\arr\sets$ is a functor then $\shF$
and $\shF^{\P}$ have the same geometric properties since $|\shF|=|\shF^{\P}|$.
Moreover if $\shQ$ is such a property then the inclusion $\shF^{\shQ}\arr\shF$
induces an isomorphism $(\shF^{\shQ})^{\P}\arr(\shF^{\P})^{\shQ}$.
\end{rem}

\begin{lem}
\label{lem:locally constructible for F and FP}Let $\cF,\shG\colon(\Sch'/S)^{\op}\arr\sets$
be functors, $\shQ$ be a geometric property for $\shG$ and $\phi\colon\shF^{\P}\arr\shG^{\P}$
be a sur epimorphism. If $\phi^{-1}(\shQ)\subseteq|\shF|$ is locally
constructible for $\shF$ then $\shQ$ is locally constructible for
$\shG$. In particular $\shF$ and $\shF^{\P}$ have the same locally
constructible properties.
\end{lem}

\begin{proof}
Let $V\arrdi A\shG$. We have to show that $A^{-1}(\shQ)\subseteq V$
is locally constructible. In particular we can assume that $V$ is
affine. Since $\phi$ is a sur epimorphism, shrinking $V$ more if
necessary, there is a commutative diagram   \[   \begin{tikzpicture}[xscale=1.5,yscale=-1.2]     \node (A0_0) at (0, 0) {$W$};     \node (A0_1) at (1, 0) {$\shF$};     \node (A0_2) at (2, 0) {$\shF^\P$};     \node (A1_0) at (0, 1) {$V$};     \node (A1_1) at (1, 1) {$\shG$};     \node (A1_2) at (2, 1) {$\shG^\P$};     \path (A0_0) edge [->]node [auto] {$\scriptstyle{B}$} (A0_1);     \path (A1_0) edge [->]node [auto] {$\scriptstyle{A}$} (A1_1);     \path (A0_2) edge [->]node [auto] {$\scriptstyle{\phi}$} (A1_2);     \path (A1_1) edge [->]node [auto] {$\scriptstyle{}$} (A1_2);     \path (A0_0) edge [->]node [auto] {$\scriptstyle{g}$} (A1_0);     \path (A0_1) edge [->]node [auto] {$\scriptstyle{}$} (A0_2);   \end{tikzpicture}   \] where
$g\colon W\arr V$ is a sur covering of schemes. By \ref{cor:finiteness for sur coverings}
we can assume that $W$ is quasi-compact and hence $g$ is quasi-compact.
As $g$ is surjective we obtain 
\[
g(B^{-1}(\phi^{-1}(\shQ)))=A^{-1}(\shQ)
\]
which is then locally constructible thanks to Chevalley's theorem
\ref{thm:Chevalley}.
\end{proof}
\begin{prop}
\label{prop:submoduli-const}Let $\cF\colon(\Sch'/S)^{\op}\arr\sets$
be a functor and $\shQ$ be a locally constructible property for $\shF$.
If $\shF$ has a strong P-moduli space $X$ which is quasi-separated
and admits a locally finite and affine open covering then $\shF^{\shQ}$
has a strong P-moduli space $Y$ which is a disjoint union of affine
schemes. If moreover $X$ is locally of finite presentation over $S$
so is $Y$.
\end{prop}

\begin{proof}
By \ref{rem:taking the geometric subfunctor commutes with taking P}
and \ref{lem:locally constructible for F and FP} we can assume that
$\shF=X$. The result follows from \ref{lem:key lemma for scheme structure on constructible sets}
and \ref{prop:locally constructible subsets as subfunctors}.
\end{proof}
\begin{defn}
\label{def:general locally constructible map} If $I$ is a set and
$\cF\colon(\Sch'/S)^{\op}\arr\sets$ is a functor, a map (or function)
$g\colon\shF\to I$ is just a map $g\colon|\shF|\to I$. The map $g$
is called \emph{locally constructible }if $g^{-1}(i)\subseteq|\shF|$
is locally constructible for all $i\in I$.
\end{defn}

\begin{prop}
\label{prop:locally constructible sur locally} Let $\phi\colon\shG\to\shF$
be a morphism of functors $(\Sch'/S)^{\op}\arr\sets$ and $g\colon\shF\to I$
be a map. If $g$ is locally constructible then so is $g\circ\phi$.
The converse holds if $\phi$ is a sur covering. In particular $\shF$
and $\shF^{\P}$ have the same locally constructible functions.
\end{prop}

\begin{proof}
The first claim follows from definition. For the converse, let $A\colon V\to\shF$
be a map from an $S$-scheme. We need to show that $g\circ A\colon V\to I$
is locally constructible. By hypothesis there exists a sur covering
$\{\phi_{j}\colon V_{j}\to V\}$ such that all $g\circ A\circ\phi_{i}$
are locally constructible. Since being a locally constructible subset
is a Zariski local property, we can assume that the sur covering has
only one element and that $V$ is affine. In other words we can assume
that $\shF=V$, $\shG=W$ is a scheme, $\phi\colon W\to V$ is a sur
covering and, by \ref{cor:finiteness for sur coverings}, that $W$
is quasi-compact. As $\phi((g\circ\phi)^{-1}(i))=g^{-1}(i)$ for all
$i$, the conclusion follows from Chevalley's theorem \ref{thm:Chevalley}.
\end{proof}
\begin{prop}
\label{prop:functor of constant maps} Let $\cF\colon(\Sch'/S)^{\op}\arr\sets$
be a functor and $g\colon\shF\to I$ be a locally constructible function.
Then the maps $\shF^{g^{-1}(i)}\to\shF$ induce a map 
\[
\bigsqcup_{i\in I}(\shF^{g^{-1}(i)})^{\P}\to\shF^{\P}
\]
where $\bigsqcup$ is the union as Zariski sheaves, which is geometrically
bijective and a luin epimorphism. If $X_{i}$ is a strong P-moduli
space for $\shF^{g^{-1}(i)}$ then $\bigsqcup_{i}X_{i}$ is a strong
P-moduli space for $\shF$.
\end{prop}

\begin{proof}
The map in the statement is well defined because $\shF^{\P}$ is a
Zariski sheaf by \ref{lem:sheaves} and it is clearly geometrically
bijective. We now show that it is an epimorphism. Consider $A\colon V\to\shF$
and set $g_{A}=A\circ g\colon V\to I$. We can assume that $V$ is
an affine scheme. By \ref{lem:key lemma for scheme structure on constructible sets}
for all $i\in I$ there exists a finitely presented monomorphism $W_{i}\to V$
whose image is $g_{A}^{-1}(i)$. It follows that $\{W_{i}\to V\}$
is a luin covering with factorizations $W_{i}\to\shF^{g^{-1}(i)}\to\shF$,
as desired.

For the last claim, set $X=\bigsqcup_{i}X_{i}$ and $h\colon X\to I$
such that $h_{|X_{i}}\equiv i$. We have $X^{h^{-1}(i)}=X_{i}$. Using
\ref{lem:when phiP is an isomorphism} we obtain
\[
X^{\P}\simeq(\bigsqcup_{i\in I}(X^{h^{-1}(i)})^{\P})^{\P}\simeq(\bigsqcup_{i\in I}X_{i}^{\P})^{\P}\simeq(\bigsqcup_{i\in I}(\shF^{g^{-1}(i)})^{\P})^{\P}\simeq\shF^{\P}
\]
\end{proof}
\begin{rem}
\label{rem:Noetherian induction} Let $X$ be a locally Noetherian
scheme and $U_{n}$ be an increasing sequence of open subsets of $X$
such that $U_{n+1}$ contains the generic points of $X\setminus U_{n}$.
Then $X=\bigcup_{n}U_{n}$. Indeed if $p\in X$ is a point, $\phi\colon\Spec(\odi{X,p})\arr X$
the structure map and $C\subseteq X$ is a closed subset then the
generic points of $\phi^{-1}(C)$ are the generic points of $C$ contained
in $\Imm(\phi)$. In particular one can assume that $X$ has finite
dimension, in which case an induction on the dimension prove the claim. 
\end{rem}

\begin{lem}
\label{lem:decomposing something locally of finite type} Let $X$
be a locally Noetherian scheme. Then there are finitely presented
immersions $U_{i}\arr X$ with $U_{i}$ affine and irreducible such
that the map
\[
\phi\colon\bigsqcup_{i}U_{i}\arr X
\]
is surjective, quasi-compact and a monomorphism. In particular it
is geometrically bijective and a $P$-isomorphism.
\end{lem}

\begin{proof}
The last claim follows from \ref{lem:when phiP is an isomorphism}.
Given a locally Noetherian scheme $X$ and a generic point $\xi$
choose an open affine subset $X_{\xi}$ of $X$ which is irreducible
and contains $\xi$ (which will be its generic point). Notice that
if $\xi$ and $\eta$ are two generic points of $X$ then $X_{\xi}\cap X_{\eta}\neq\emptyset$
implies $\xi=\eta$. Set
\[
V(X)=\bigsqcup_{\xi\text{ generic point of }X}X_{\xi}\text{ and }Z(X)=X\setminus V(X)
\]
So that $V(X)$ is open and $Z(X)$ is closed. The latter will be
though of as a closed subscheme with reduced structure. Since $X$
is locally Noetherian the map $Z(X)\arr X$ is a closed immersion
of finite type. By induction set $V_{n+1}(X)=V(Z_{n}(X))$, $Z_{n+1}(X)=Z(Z_{n}(X))$,
$V_{0}(X)=\emptyset$ and $Z_{0}(X)=X$. By construction all maps
$Z_{n}(X)\arr X$ are closed immersion of finite type and $V_{n}(X)\arr X$
are immersion of finite type. Moreover 
\[
X\setminus Z_{n}(X)=\bigsqcup_{k=0}^{n}V_{k}(X)
\]
as sets. In conclusion the map
\[
\bigsqcup_{n}V_{n}(X)\arr X
\]
 is a monomorphism by construction and it is surjective by \ref{rem:Noetherian induction}.
It remains to show that it is quasi-compact. So let $U\subseteq X$
be a quasi-compact open subset. Since the union of $(X\setminus Z_{n}(X))\cap U$
covers $U$, the previous sequence must stabilize. Moreover since
$Z_{n}(X)\cap U$ is quasi-compact and Noetherian, it follows that
$V_{n+1}(X)\cap U$ is a finite disjoint union of its irreducible
components. This ends the proof.
\end{proof}
We are going to introduce some notation to explain next Lemma \ref{lem:strong moduli of ind coarse},
which is a key ingredient in the proof of Theorem \ref{thm:main1}.
A direct system of stacks $\{\stZ_{n}\}_{n\in\NN}$ is a chain of
stacks $\stZ_{0}\to\cdots\to\stZ_{n}\to\stZ_{n+1}\to\cdots$ (see
\cite[Appendix A]{Tonini:2017qr}). Such a sequence always admits
a stack $\stZ_{\infty}=\colim_{n}\stZ_{n}$ as colimit (see \cite[Proposition A.1 and Proposition A.5]{Tonini:2017qr},
\cite[Remark A.3]{Tonini2016}). Moreover for an affine scheme $U$
any object $U\to\stZ_{\infty}$ is induced by some $U\to\stZ_{n}$
(see \cite[just before Proposition A.2]{Tonini:2017qr}).

Assume that all $\stZ_{n}$ are separated Deligne-Mumford stacks of
finite type over a field $k$, so that they admit coarse moduli spaces
$\stZ_{n}\to\overline{\stZ_{n}}$. As a consequence of \cite[Lemma 3.2]{Tonini:2017qr}
we have the following. The colimit $\stZ_{\infty}\to\colim_{n}(\overline{\stZ_{n}})=\overline{\stZ_{\infty}}$
of the coarse moduli maps $\stZ_{n}\to\overline{\stZ_{n}}$ is a coarse
ind-algebraic space map in the sense of \cite[Definition 3.1]{Tonini:2017qr}.
Moreover if the transition maps $\stZ_{n}\to\stZ_{n+1}$ are finite
and universally injective then so are the maps $\overline{\stZ_{n}}\to\overline{\stZ_{n+1}}$:
they are universally injective, thus quasi-finite, by \cite[Lemma 3.2]{Tonini:2017qr},
they are proper because so are the coarse moduli maps $\stZ_{n}\to\overline{\stZ_{n}}$.
By \cite[Proposition 2.9]{Tonini:2017qr} finite and universally injective
is the same as a composition of a finite universal homeomorphism and
a closed immersion.
\begin{lem}
\label{lem:strong moduli of ind coarse} Let $\{\stZ_{n}\}_{n\in\NN}$
be a direct system of separated Deligne-Mumford stacks of finite type
over $k$ with finite and universally injective transition maps and
colimit $\stZ$. Then there are affine varieties $\{Y_{i}\}_{i\in\N}$
and a map 
\[
\coprod_{i}Y_{i}\arr\overline{\stZ}
\]
where $\overline{(-)}$ denotes the corresponding ind-coarse moduli
space, which is geometrically bijective and an epimorphism in the
sur topology, so that $\coprod_{i}Y_{i}$ is a strong P-moduli space
for $\overline{\stZ}$. Moreover the functor of isomorphism classes
of $\stZ$, its Zariski, étale and fppf sheafifications all have the
same strong P-moduli space of $\overline{\stZ}$.
\end{lem}

\begin{proof}
Set $\stU_{n+1}=\stZ_{n+1}\setminus\stZ_{n}$. It is easy to see that
$\overline{\stZ_{n}}\amalg\overline{\stU_{n+1}}\arr\overline{\stZ_{n+1}}$
is geometrically bijective. Since $\overline{\stZ}$ is the limit
of the $\overline{\stZ_{n}}$ the induced map $\coprod_{n}\overline{\stU_{n}}\arr\overline{\stZ}$
is geometrically bijective. Moreover it is an epimorphism in the sur
topology because any map $U\to\overline{\stZ}$ from an affine scheme
factors through some $\overline{\stZ_{n}}$. Each $\overline{\stU_{n}}$
is an algebraic space of finite type over $k$. By \ref{lem:when phiP is an isomorphism}
and \ref{rem:strong moduli of algebraic spaces} the first part of
the statement follows.

Now denote by $\shF$ the functor of isomorphism classes of $\stZ$
and by $\shF^{sh}$ its sheafification for some of the topologies
in the statement or $\shF$ itself. The map $\shF^{sh}\arr\overline{\stZ}$
is geometrically bijective by definition of coarse ind-algebraic space
(see \cite[Definition 3.1]{Tonini:2017qr}). It is also an epimorphism
in the sur topology: a map $V\arr\overline{\stZ}$ from a scheme factors
Zariski locally through $\overline{\stZ_{n}}$, sur locally through
$\stZ_{n}$ and therefore through $\shF^{sh}$. From \ref{lem:when phiP is an isomorphism}
we conclude that $(\shF^{sh})^{\P}\simeq\overline{\stZ}^{\P}$. 
\end{proof}

\subsection{P-schemes locally of finite type over a locally Noetherian scheme.}

We fix a locally Noetherian scheme $S$ as base.
\begin{rem}
\label{rem:ubi is quasi-compact} By \ref{cor:ubi is quasi-compact}
a geometrically bijective map $f\colon X\arr Y$ between schemes locally
of finite type over $S$ is quasi-compact and, therefore, of finite
type.
\end{rem}

\begin{lem}
\label{lem:P-morphism in the simple case} Let $X$ and $Y$ be schemes
locally of finite type over $S$ and let $f\colon Y\arr X$ be a P-morphism
over $S$. Then
\begin{itemize}
\item there exists a geometrically bijective morphism $Z\arr Y$ of finite
type such that the composite $P$-morphism $Z\arr Y\arr X$ is induced
by a scheme morphism $Z\longrightarrow X$;
\item the map $f$ is a P-isomorphism if and only if $f$ is geometrically
bijective.
\end{itemize}
\end{lem}

\begin{proof}
By \ref{lem:decomposing something locally of finite type} we can
assume that $S$ is affine and that $X$ and $Y$ are disjoint unions
of affine schemes. In particular $X$ and $Y$ are separated over
$S$. The first statement follows from \ref{cor:best ubi covering}
and \ref{lem:P-morphisms to finitely presented stuff}, the second
from \ref{lem:criterion one for Piso}.
\end{proof}
\begin{cor}
\label{cor:P-isomorphisms in finite type case} Two schemes $X$ and
$Y$ locally of finite type over $S$ are $P$-isomorphic if and only
if there exist a scheme $Z$ and geometrically bijective maps of finite
type $Z\arr X$ and $Z\arr Y$.
\end{cor}

\begin{proof}
The ``if part'' follows directly from \ref{lem:P-morphism in the simple case}.
Indeed $Z$ is locally of finite type over $S$ because, for instance,
$Z\to X$ is of finite type. Thus $Z\to X$, $Z\to Y$ are P-isomorphisms
because they are geometrically bijective.

Let's focus on the ``only if part''. Let $f\colon Y\to X$ be a
P-isomorphism. By \ref{lem:P-morphism in the simple case} there is
a geometrically bijective morphism $Z\to Y$ of finite type such that
the composition $Z\to Y\to X$ is induced by a scheme morphism $g\colon Z\to X$.
In particular $Z$ is locally of finite type over $S$. Moreover $g$
is geometrically bijective and, by \ref{rem:ubi is quasi-compact},
of finite type.
\end{proof}
\begin{lem}
\label{lem:P-iso of locally constructible} Let $X$ be a locally
Noetherian scheme and $C\subseteq X$ be a locally constructible subset.
Then there exists a monomorphism $Z\arr X$ of finite type with image
$C$. Moreover if $Z'\arr X$ is another map which is geometrically
injective, locally of finite type and has image $C$ then $Z'$ and
$Z$ are $P$-isomorphic over $X$.
\end{lem}

\begin{proof}
The existence follows from \ref{lem:key lemma for scheme structure on constructible sets}
and \ref{lem:decomposing something locally of finite type}. For the
last statement notice that the projections $Z\times_{X}Z'\rightrightarrows Z,Z'$
are P-isomorphisms thanks to \ref{lem:P-morphism in the simple case}.
\end{proof}
\begin{defn}
In the situation of Lemma \ref{lem:P-iso of locally constructible}
we will say that a scheme is \emph{P-isomorphic to} $C$ if it is
P-isomorphic to $Z$.
\end{defn}

We conclude the section by an useful result for schemes over a field.
\begin{lem}
\label{lem:dimension and universally injective maps} Let $X$ and
$Y$ be schemes locally of finite type over $k$, $f\colon X\arr Y$
be a geometrically injective map and $x\in X$. Then, for every point
$x\in X$,
\[
\dim\overline{\{x\}}=\dim\overline{\{f(x)\}}=\degtr k(x)/k
\]
where $\degtr$ denotes the transcendence degree. Moreover $\dim X\leq\dim Y$
and the equality holds if $f$ is also surjective. In particular two
schemes locally of finite type over $k$ and P-isomorphic have the
same dimension.
\end{lem}

\begin{proof}
The last two statements are a consequence of the first one and \ref{cor:P-isomorphisms in finite type case}.
The equality $\dim\overline{\{x\}}=\degtr k(x)/k$ is \cite[Tag 02JX]{stacks-project}.
The equality $\dim\overline{\{x\}}=\dim\overline{\{f(x)\}}$ instead
follows from the fact that, if $y=f(x)$, then $k(x)/k(y)$ is finite:
the fiber map $X\times_{Y}k(y)\arr\Spec k(y)$ is non-empty, geometrically
injective and locally of finite type, which easily implies that $X\times_{Y}k(y)$
is the spectrum of a local and finite $k(y)$-algebra.
\end{proof}

\subsection{Quotients by P-actions of finite groups}

In what follows $G$ will denote a finite group.
\begin{defn}
\label{def:Pactions} Let $X$ be an $S$-scheme. A \emph{P-automorphism}
of $X$ is a P-morphism $f\colon X\longrightarrow X$ which is invertible,
that is, there exists a P-morphism $f'\colon X\longrightarrow X$
with $f\circ f'=f'\circ f=\id_{X}$. A \emph{P-action} of a finite
group $G$ on $X$ means a group homomorphism $G\longrightarrow\Aut_{S}^{P}(X)$,
where $\Aut_{S}^{P}(X)$ is the group of P-automorphisms of $X$.
A P-morphism $g\colon X\to Y$ between two $S$-schemes with a P-action
of $G$ is \emph{equivariant }it is so in the category of P-schemes
over $S$.

When we are given a P-action of a group $G$ on a scheme $X$, a \emph{geometric
P-quotient} is a P-morphism $\pi\colon X\longrightarrow W$ of $S$-schemes
such that: 
\begin{enumerate}
\item the map $\pi$ is $G$-invariant, that is, for every $g\in G$, $\pi\circ g=\pi$, 
\item the map $\pi$ is universal among $G$-invariant P-morphisms, that
is, if $\pi'\colon X\longrightarrow W'$ is another $G$-invariant
P-morphism of $S$-schemes, then there exists a unique P-morphism
$h\colon W\longrightarrow W'$ such that $h\circ\pi=\pi'$,
\item for each algebraically closed field $K$ over $S$, the map $X(K)/G\longrightarrow W(K)$
is bijective. 
\end{enumerate}
A P-morphism $\pi\colon X\longrightarrow W$ of $S$-schemes is a
\emph{strong P-quotient }if it is $G$-invariant and the induced map
$X^{P}/G\longrightarrow W^{P}$ is a strong P-moduli space, where
$X^{P}/G$ is the functor $U\longmapsto X^{P}(U)/G$. 
\end{defn}

\begin{rem}
Recall that for $S$-schemes $X$ and $Y$ one has 
\[
\Hom_{S}^{\P}(X,Y)\simeq\Hom_{S}(X^{\P},Y^{\P})
\]
by \ref{prop:remarks on P and F}, $5)$, more precisely $(-)^{\P}\colon\PSch/S\to\Sh_{\text{sur}}(\Sch'/S)$
is fully faithful by \ref{cor:Pschemes into sheaves}. In particular:
a $P$-action of $G$ on $X$ is just an action of $G$ on $X^{\P}$;
a $G$-invariant map $X\arr W$ is a $G$-invariant map $X^{\P}\arr W^{\P}$,
that is a map $X^{\P}/G\arr W^{\P}$; an equivariant P-morphism $Y\to X$
is an equivariant morphism $X^{\P}\to Y^{\P}$. Moreover it follow
easily that $X\arr W$ is a geometric $P$-quotient (resp. strong
$P$-quotient) if and only if $X^{\P}/G\arr W^{\P}$ is a $P$-moduli
space (resp. a strong $P$-moduli space).
\end{rem}

\begin{prop}
Let $X$ be an $S$-scheme with a $P$-action of $G$ and $X\arr W$
be a $G$-invariant $P$-morphism over $S$ such that $X(K)/G\arr W(K)$
are bijective  for all algebraically closed fields $K$ over $S$.
If $X\arr W$ is an epimorphism in the sur topology then $X\arr W$
is a strong $P$-quotient. This is the case, for example, if $X\arr W$
is a locally of finite presentation map of $S$-schemes.
\end{prop}

\begin{proof}
Set $F=X^{\P}/G$. By hypothesis the map $F\arr W^{\P}$ is geometrically
bijective, that is, by \ref{lem:when phiP is an isomorphism}, the
map $F^{\P}\arr W^{\P}$ is a monomorphism. The previous map is an
isomorphism, that is $X\arr W$ is a strong $P$-quotient, if and
only if $F\arr W^{\P}$ is an epimorphism in the sur topology. This
is true if $X\arr W$ is an epimorphism as well. Notice that $X\arr W$
is surjective because the map $X(K)\arr X(K)/G\arr W(K)$ is so for
all algebraically closed fields $K$ over $S$. Therefore if $X\arr W$
is locally of finite presentation then this map is a sur covering.
\end{proof}
\begin{cor}
\label{cor:geometric and strong quotients} Let $X$ be an $S$-scheme
with an (usual) action of $G$. If $X\arr W$ is a $G$-invariant
map of $S$-schemes, it is locally of finite presentation and $X(K)/G\arr W(K)$
is an isomorphism for all algebraically closed fields $K$, then it
is also a strong $P$-quotient. In particular geometric quotients
are strong $P$-quotients.
\end{cor}

\begin{lem}
\label{lem:lifting P-morphisms} Consider a diagram  \[   \begin{tikzpicture}[xscale=1.5,yscale=-1.2]     \node (A0_1) at (1, 0) {$Z$};     \node (A1_0) at (0, 1) {$X$};     \node (A1_1) at (1, 1) {$Y$};     \path (A1_0) edge [->,dashed]node [auto] {$\scriptstyle{\tilde f}$} (A0_1);     \path (A1_0) edge [->]node [auto] {$\scriptstyle{f}$} (A1_1);     \path (A0_1) edge [->]node [auto] {$\scriptstyle{u}$} (A1_1);   \end{tikzpicture}   \] of
$S$-schemes where $f$ is a $P$-morphism and $u$ is a locally finitely
presented and geometrically injective map. If $f(X)\subseteq u(Z)$
as sets then there exists a unique dashed $P$-morphism $\tilde{f}$
making the above diagram commutative in $\PSch/S$.
\end{lem}

\begin{proof}
Since $u$ is locally of finite type we have that $f\colon X_{F}\arr Y_{F}$
has value in $Z_{F}\subseteq Y_{F}$. We just have to show that $X_{F}\arr Z_{F}$
is a $P$-morphism. In particular we can assume that $f$ is induced
by a map of schemes. In this case we obtain a map $X\times_{Y}Z\arr X$
which is locally of finite presentation and, by hypothesis, surjective.
Thus it is a sur covering of $X$ and the map $X\times_{Y}Z\arr Z$
lifts $(X\times_{Y}Z)_{F}\arr X_{F}\arr Z_{F}$.
\end{proof}
\begin{lem}
\label{lem:open-dense} Let $X$ be a scheme of finite type over a
field $k$ endowed with a P-action of a finite group $G$ and let
$U\subset X$ be an open subset with $\dim(X\setminus U)<\dim X$.
Then there exists an open subset $V\subset U$ with $\dim(X\setminus V)<\dim X$
which is $G$ invariant, a finite universal homeomorphism $h\colon V'\arr V$
and an action of $G$ on $V'$ making $h$ equivariant. 
\end{lem}

\begin{proof}
Notice that the condition $\dim(X\setminus U)<\dim X$ just means
that $U$ meets the irreducible components of $X$ of maximal dimension
$\dim X$. From \ref{lem:P-morphism in the simple case}, there exists
a geometrically bijective map $Z_{g}\arr U$ of finite type such that
$Z_{g}\arr U\arrdi gX$ is induced by a scheme morphism. Taking the
fiber products of the $Z_{g}$ over $U$ we can find a common map
$h\colon Z\arr U$. Call $h_{g}\colon Z\arr X$ the lifting of $U\arrdi gX$,
with $h_{\id}=h$.

We first show that $G$ permutes the generic points of the irreducible
components of $X$ of dimension $d=\dim X$. If $\xi$ is a generic
point of such a component, then $\xi\in U$, $g(\xi)=h_{g}(h^{-1}(\xi))$
and using \ref{lem:dimension and universally injective maps}, it
follows that $d=\dim\overline{\{\xi\}}=\dim\overline{\{g(\xi)\}}$.
Since $\dim X=d$ we can also conclude that $g(\xi)$ is a generic
point. 

The maps $h_{g}\colon Z\arr X$ are quasi-compact, quasi-separated
and geometrically injective. By \cite[Tag 02NW]{stacks-project} there
exists an open dense subset $W$ of $U$ such that $h_{g}^{-1}(W)\arr W$
is finite for all $g$. Set $W'=h^{-1}(W)$. In particular $h\colon W'\arr W$
is a finite universal homeomorphism. Notice that $h(h_{g}^{-1}(W)\cap W')=W\cap g^{-1}(W)$
as sets and it is an open subset of $W$. Consider $V:=\bigcap_{g\in G}g(W)$,
which is open in $W$ and set $V':=h^{-1}(V)\longrightarrow V$, which
is a finite universal homeomorphism. Notice that $V$ contains the
generic points of the irreducible components of maximal dimension.
Therefore $\dim(X\setminus V)<\dim X$. Moreover the composition $V'\subseteq W'\arrdi{h_{g}}X$,
which set-theoretically is $V'\arr V\arrdi gX$, factors through $V$
and $h_{g}\colon V'\arr V$ is surjective. Since this map is a restriction
of the finite and geometrically injective map $h_{g}^{-1}(W)\arr W$,
we can conclude that $h_{g}\colon V'\arr V$ is a finite universal
homeomorphism.

We now modify $V'$ in order to define an action on it. Notice that
if $\widetilde{V}$ is an open subset of $V$ with $\dim(X\setminus\widetilde{V})<\dim X$,
by discussion above it always contains a $G$-invariant open with
the same property and we can always replace $V$ by it. Moreover we
can always assume $X=V$. In conclusion we can shrink as much as we
want around the generic points of $X$ of maximal dimension. In particular
we can assume that $V=X$ and $V'$ are a disjoint union of affine
integral varieties of the same dimension.

Let $G(X)$ the generic points of $X$ and for $\xi\in G(X)$ let
$\eta_{\xi}$ the generic point of $V'$ mapping to $\xi$. For all
$\xi\in G(X)$ set also $K_{\xi}$ for the perfect closure of $k(\xi)$.
Recall that if $L/k(\xi)$ is a purely inseparable extension then
there exists a unique $k(\xi)$ linear map $L\arr K_{\xi}$. Since
$V'\arr X$ is a finite universal homeomorphism it follows that $k(\xi)\arr k(\eta_{\xi})$
is finite and purely inseparable. So we can assume $k(\eta_{\xi})\subseteq K_{\xi}$.
We have that $G$ permutes $G(X)$ and, since $h_{g}\colon V'\arr X$
is a finite universal homeomorphism, it also induces a finite purely
inseparable extension $k(g(\xi))\arr k(\eta_{\xi})$. In particular
there exists a unique map $\phi_{g,\xi}$ making the following diagram
commutative:   \[   \begin{tikzpicture}[xscale=2.2,yscale=-1.2]     \node (A0_1) at (1, 0) {$k(\eta_\xi)$};     \node (A0_2) at (2, 0) {$K_\xi$};     \node (A1_0) at (0, 1) {$k(g(\xi))$};     \node (A1_1) at (1, 1) {$k(\eta_{g(\xi)})$};     \node (A1_2) at (2, 1) {$K_{g(\xi)}$};     \path (A1_0) edge [->]node [auto] {$\scriptstyle{h_g}$} (A0_1);     \path (A1_0) edge [->]node [auto] {$\scriptstyle{}$} (A1_1);     \path (A0_1) edge [->]node [auto] {$\scriptstyle{}$} (A0_2);     \path (A0_2) edge [->]node [auto] {$\scriptstyle{\phi_{g,\xi}}$} (A1_2);     \path (A1_1) edge [->]node [auto] {$\scriptstyle{}$} (A1_2);   \end{tikzpicture}   \] We
claim that the two maps $\phi_{ab,\xi},\phi_{a,b(\xi)}\circ\phi_{b,\xi}\colon K_{\xi}\arr K_{ab(\xi)}$
are the same map. Let $\alpha,\beta\colon\Spec K_{ab(\xi)}\arr\Spec K_{\xi}$
be the corresponding maps. By hypothesis they coincide as $P$-morphisms
if composed with $\Spec K_{\xi}\arr X$. If $\overline{K}$ is an
algebraic closure of $K_{ab(\xi)}$ then the two maps
\[
\Spec\overline{K}\arr\Spec K_{ab(\xi)}\rightrightarrows\Spec K_{\xi}\arr\Spec k(\xi)
\]
 coincide. Using the usual properties of purely inseparable extensions
and the perfect closure we can conclude that $\alpha=\beta$. In particular
all maps $\phi_{g,\xi}$ are isomorphisms. If we set $\tilde{K}_{\xi}$
as the composite of all extensions $\phi_{g,\xi}^{-1}(k(\eta_{g(\xi)}))$
it follows that $\tilde{K}_{\xi}/k(\eta_{\xi})$ is finite and purely
inseparable and $\phi_{g,\xi}$ restricts to an isomorphism $\tilde{K}_{\xi}\arr\tilde{K}_{g(\xi)}$.
If $V'_{\xi}$ is the irreducible component of $\eta_{\xi}$ we can
find an open dense $U_{\xi}$ and a finite universal homeomorphism
$U'\arr U_{\xi}$ with $U'$ integral and fraction field $\tilde{K}_{\xi}$.
Shrinking $X$ we can assume $k(\eta_{\xi})=\tilde{K}_{\xi}$. The
map $\phi_{g^{-1},g(\xi)}$ yield a generic map $\psi_{g,\xi}\colon V'_{\xi}\arr V'_{g(\xi)}$
and shrinking again $X$ we can assume it is defined everywhere and,
more generally, that it defines an action of $G$ on $V'$.

The maps $V'\arrdi{\psi_{g}}V'\arrdi hX$ and $V'\arrdi{h_{g}}X$
coincide in the generic points and therefore they are generically
the same because $V'$ is reduced. Again shrinking $X$ we can assume
they coincide. But this exactly means that the $P$-action of $G$
on $V'$ obtained conjugating the $P$-isomorphism $V'\arr X$ is
induced by the maps $\psi_{g}$ on $V'$. By \ref{prop:remarks on P and F},
(7) we can conclude that the collection of maps $\{\psi_{g}\}_{g}$
defines a ``genuine'' action of $G$ on $V'$, which ends the proof.
\end{proof}
\begin{prop}
\label{lem:decomposing a P-action} Let $X$ be a scheme locally of
finite type over a field $k$ and with a $P$-action of a finite group
$G$. Then there exist a locally of finite type scheme $Y$ with an
action of $G$, a geometrically bijective map $Y\arr X$ of finite
type which is $G$-equivariant and a decomposition of $Y=\bigsqcup_{i}Y_{i}$
into $G$-invariant open affine subsets.
\end{prop}

\begin{proof}
From \ref{lem:decomposing something locally of finite type} and \ref{lem:lifting P-morphisms}
we can assume $X=\bigsqcup X_{q}$ where the $X_{q}$ are integral
schemes of finite type over $k$. From \ref{lem:P-morphism in the simple case},
there exists a geometrically bijective map $\phi\colon Z_{g}\arr X$
of finite type such that $Z_{g}\arr X\arrdi gX$ is induced by a scheme
morphism $h_{g}\colon Z_{g}\arr X$. Taking the fiber products of
the $Z_{g}$ over $X$ we can find a common map $\phi\colon Z\arr X$.
Since $g(X_{q})=h_{g}(\phi^{-1}(X_{q}$)), this is a locally constructible
set of $X$. Moreover since $X_{q}$ is quasi-compact and $\phi$
is of finite type, $g(X_{q})$ is contained in a quasi-compact open
of $X$. In particular $Z_{q}=\bigcup_{g}g(X_{q})$ is a locally constructible
subset of $X$ contained in a quasi-compact open subset. Moreover
it is $G$-invariant. We use the notation in \ref{nota:splitting a covering by constructible sets}
with $I$ the index set of the $Z_{q}$. Let $q\in I$ and consider
indexes $Z_{q}\subseteq X_{q_{1}}\amalg\cdots\amalg X_{q_{l}}$. We
claim that $Z_{q}\cap Z_{q'}\neq\emptyset$ implies that $q'=q_{i}$
for some $i$. From this and \ref{nota:splitting a covering by constructible sets}
it will follow that, for $J\subseteq I$ finite, $Z_{J}$ is locally
constructible and 
\[
X=\bigsqcup_{J\subseteq I\text{ finite}}Z_{J}
\]
as sets. If $Z_{q}\cap Z_{q'}\neq\emptyset$ there exist $g,h\in G$
such that $g(X_{q})\cap h(X_{q'})\neq\emptyset$, that is $\emptyset\neq h^{-1}g(X_{q})\cap X_{q'}\subseteq Z_{q}\cap X_{q'}$,
from which the claim follows.

For all $J$ finite, since $Z_{J}$ is locally constructible, we have
a monomorphism $Y_{J}\arr X$ of finite type onto $Z_{J}$ by \ref{lem:key lemma for scheme structure on constructible sets}.
Since $Z_{J}$ is contained in a quasi-compact open of $X$ it follows
that $Y_{J}$ is quasi-compact, that is of finite type. By construction
the $Z_{J}$ are $G$-invariant and, by \ref{lem:lifting P-morphisms},
we can lift the P-action of $G$ on $X$ to a P-action of $G$ on
$Y_{J}$. 

The argument above shows that we can replace $X$ by a scheme of finite
type. We can also assume $X$ reduced and, by \ref{lem:decomposing something locally of finite type},
also separated. Consider the open $V$ and the map $h\colon V'\arr V$
obtained from \ref{lem:open-dense}. By a dimension argument and an
induction on $\dim X$ we can assume $V=X$ and that $G$ has a genuine
action on $X$ inducing the $P$-action. Consider a dense affine open
subset $W$ of $X$ and replacing it by $\bigcap_{g}g(W)$ so that
it is also $G$-invariant. Again since $\dim(X\setminus W)<\dim X$
we can assume $X=W$ and we are done.
\end{proof}
\begin{thm}
\label{prop:quotient by finite group } Let $X$ be a scheme (locally)
of finite type over a field $k$ endowed with a P-action of a finite
group $G$. Then $X$ has a strong P-quotient $X\longrightarrow Y$
with $Y$ (locally) of finite type over $k$. Moreover if $X$ is
P-isomorphic to a countable disjoint union of affine $k$-varieties
then so is the strong P-quotient $Y$.
\end{thm}

\begin{proof}
Notice that if $U=\bigsqcup_{n\in\N}U_{n}$ is P-isomorphic to $V=\bigsqcup_{i\in I}V_{i}$
where $V_{i}$ and $U_{n}$ are schemes of finite type over $k$ with
$V_{i}\neq\emptyset$ then $I$ is at most countable. Indeed there
exist geometrically bijective maps of finite type $\phi\colon Z\arr U$
and $\psi\colon Z\arr V$ thanks to \ref{cor:P-isomorphisms in finite type case}.
Thus one can assume $Z=U=V$ and notice that the sets $\{i\in I\st U_{n}\cap V_{i}\neq\emptyset\}$
are finite and cover $I$. Thanks to the previous observation and
by \ref{lem:decomposing a P-action} we can assume $X=\Spec A$ affine
and that the $P$-action of $G$ on $X$ is actually an action. Then
$X\arr X/G=\Spec(A^{G})$ is a geometric quotient and $A^{G}$ is
of finite type over $k$. By \ref{cor:geometric and strong quotients}
the map $X\arr X/G$ is a strong $P$-quotient.
\end{proof}

\section{Motivic integration on schemes locally of finite type\label{sec:Motivic-integration}}

In this section we construct a modified Grothendieck ring using the
theory of P-schemes.
\begin{defn}
\label{def:Grothedieck rings}The \emph{modified Grothendieck ring
of varieties}, $K_{0}^{\modified}(\Var/k)$, is the free abelian group
generated by the P-isomorphism classes of $k$-varieties modulo the
relation $[X]=[U]+[V]$ if $X$ and $U\amalg V$ are P-isomorphic.
The product structure is given by $[X][Y]:=[X\times Y]$.
\end{defn}

In particular, the usual scissor relation holds: if $Y\subset X$
is a closed subvariety, then $[X]=[Y]+[X\setminus Y]$. Moreover,
if $X$ and $Y$ are P-isomorphic, then $[X]=[Y]$.
\begin{defn}
We denote by $\LL$ the class of an affine line $[\AA_{k}^{1}]$ in
$K_{0}^{\modified}(\Var/k)$. We define $\cM_{k}^{\modified}$ to
be the localization of $K_{0}^{\modified}(\Var)$ by $\LL$. For a
positive integer $l$, we define $\cM_{k}^{\modified,l}$ to be $\cM_{k}^{\modified}[\LL^{1/l}]=\cM_{k}^{\modified}[x]/(x^{l}-\LL)$.
In this ring, we have fractional powers $\LL^{r}$, $r\in\frac{1}{l}\ZZ$
of $\LL$. We then define a completion $\hat{\cM}_{k}^{\modified,l}$
of $\cM_{k}^{\modified,l}$as follows. Let $F_{m}\subset\cM_{k}^{\modified,l}$
be the subgroup generated by the elements $[X]\LL^{r}$ with $\dim X+r\le-m$.
We define
\[
\hat{\cM}_{k}^{\modified,l}:=\varprojlim_{m}\cM_{k}^{\modified,l}/F_{m},
\]
which inherits the ring structure since $F_{m}F_{n}\subset F_{m+n}$.
When $l=1$, we abbreviate $\cM_{k}^{\modified,l}$ and $\hat{\cM}_{k}^{\modified,l}$
to $\cM_{k}^{\modified}$ and $\hat{\cM}_{k}^{\modified}$ respectively.
\end{defn}

Recall that a P-morphism $X\arr Y$ of schemes induces a map $|X|\arr|Y|$
on the set of points.

\begin{defn}
\label{def:locally constructible and integrable functions} Let $X$
be a scheme locally of finite type over $k$, $l\in\Z\setminus\{0\}$
and $f\colon X\arr\frac{1}{l}\Z$ be a function, that is a map of
sets from the set of points $|X|$ of $X$ to $\frac{1}{l}\Z$. The
map $f$ is called\emph{ integrable }if there are non-empty schemes
$\{X_{i}\}_{i\in I}$ of finite type over $k$ and a P-isomorphism
$\phi\colon\coprod_{i}X_{i}\arr X$ such that $f\circ\phi$ is constant
on all $X_{i}$ and, for all $n\in\Z$, there are at most finitely
many $i\in I$ such that $\dim X_{i}+f(\phi(X_{i}))>n$.

We define the \emph{integral }$\int_{X}\LL^{f}\in\hat{\cM}_{k}^{\modified,l}\cup\{\infty\}$
of a function $f\colon X\arr\frac{1}{l}\Z$ as follows. If $f$ is
integrable,
\[
\int_{X}\LL^{f}:=\sum_{i\in I}[X_{i}]\LL^{f(\phi(X_{i}))}\in\hat{\cM}_{k}^{\modified,l}.
\]
Otherwise$\int_{X}\LL^{f}:=\infty$.
\end{defn}

Notice that, if we follow the usual convention that $\dim\emptyset=-\infty$,
in the definition of integrability and of integrals we don't have
to assume that the schemes $X_{i}$ are non empty.

The following lemma shows that the notion of integrability and the
integral itself do not depend on the choice of the $k$-schemes $X_{i}$.
\begin{lem}
\label{lem:independence integrability} Let $X$ be a scheme locally
of finite type over $k$, $l\in\Z\setminus\{0\}$ and $f\colon X\arr\frac{1}{l}\Z$
be a function. Let $\{Y_{j}\}_{j\in J}$ be non-empty schemes of finite
type over $k$ and $\phi\colon Y=\coprod_{j}Y_{j}\arr X$ be a P-isomorphism
such that $f\circ\phi$ is constant on all $Y_{j}$. If $f$ is integrable,
then for each $n\in\frac{1}{l}\ZZ$, there are at most finitely many
$j\in J$ such that $\dim Y_{j}+f(Y_{j})>n$ and 
\[
\int_{X}\LL^{f}=\sum_{j\in J}[Y_{j}]\LL^{f(\phi(Y_{j}))}\in\hat{\cM}_{k}^{\modified,l}.
\]
\end{lem}

\begin{proof}
Following the notation of Definition \ref{def:locally constructible and integrable functions}
we can assume $X=\coprod_{i}X_{i}$. By \ref{cor:P-isomorphisms in finite type case}
there exist a scheme $Z$ and geometrically bijective maps of finite
type $\alpha\colon Z\arr\coprod_{i}X_{i}$, $\beta\colon Z\arr\coprod_{j}Y_{j}$.
In particular $\alpha^{-1}(X_{i})$ and $\beta^{-1}(Y_{j})$ are of
finite type and those maps preserve dimension thanks to \ref{lem:dimension and universally injective maps}.
We can therefore assume $Z=X=Y$. Set
\[
I_{n}=\{i\in I\st f(X_{i})+\dim X_{i}>n\}\text{ and }J_{n}=\{j\in J\st f(Y_{j})+\dim Y_{j}>n\}.
\]
Given $j\in J$ take a generic point $\eta_{j}\in Y_{j}$ with $\dim\overline{\{\eta_{j}\}}=\dim Y_{j}$
and let $s_{j}\in I$ be such that $X_{s_{j}}$ contains the point
$\eta_{j}$. We have $f(X_{s_{j}})=f(Y_{j})$ and, by \ref{lem:dimension and universally injective maps},
$\dim Y_{j}\leq\dim X_{s_{j}}$. In particular $s\colon J\arr I$
maps $J_{n}$ into $I_{n}$ and, in order to show that $J_{n}$ is
finite, it is enough to show that $s$ has finite fibers. The result
follows from
\[
X_{i}=\bigsqcup_{j\in J}X_{i}\cap Y_{j}
\]
and the fact that $s_{j}=i$ implies that $X_{i}\cap Y_{j}\neq\emptyset$. 

For the last equality, it is enough to use the (finite) sums
\[
[X_{i}]=\sum_{j}[X_{i}\cap Y_{j}]\text{ and }[Y_{j}]=\sum_{i}[X_{i}\cap Y_{j}]
\]
in $K_{0}^{\modified}(\Var/k)$ and that, if $X_{i}\cap Y_{j}\neq\emptyset$
then $f(X_{i})=f(X_{i}\cap Y_{j})=f(Y_{j})$. 
\end{proof}
\begin{defn}
Let $\shF\colon\Sch'/k\arr\sets$ be a functor with a scheme locally
of finite type $X$ as strong P-moduli space and $f\colon\shF\arr\frac{1}{l}\Z$
be a function, which is induced by $f_{X}\colon X\to\frac{1}{l}\Z$
(see \ref{def:general locally constructible map}). The map $f$ is
called \emph{integrable }if $f_{X}$ is so. Moreover we set $\int_{\shF}\LL^{f}=\int_{X}\LL^{f_{X}}$.

If $Y$ is a scheme locally of finite type over $k$ and $C\subseteq Y$
a locally constructible subset a function $f\colon C\arr\frac{1}{l}\Z$
is just a function of sets $|C|\arr\frac{1}{l}\Z$. We define \emph{constructibility}
and \emph{integrability} for $f\colon C\arr\frac{1}{l}\Z$ as the
ones for $f\colon X\arr\frac{1}{l}\Z$, where $X$ is a scheme P-isomorphic
to $C$. Moreover we set $\int_{C}\LL^{f}=\int_{X}\LL^{f}$.
\end{defn}

Notice that, by \ref{prop:locally constructible subsets as subfunctors},
in the above definition the second definition is a particular case
of the previous one.
\begin{prop}
Let $f\colon X\arr\frac{1}{l}\Z$ be a function from a scheme locally
of finite type over $k$. Then $f$ is integrable if and only if the
following three conditions are satisfied: (1) $f$ is bounded above,
(2) for all $n\in\frac{1}{l}\Z$ the set $f^{-1}(n)$ is locally constructible
and P-isomorphic to a scheme of finite type over $k$ and (3)
\[
n-\dim(f^{-1}(-n))\arr+\infty\text{ for }\frac{1}{l}\Z\ni n\arr+\infty
\]
where we use the usual convention $\dim\emptyset=-\infty$.
\end{prop}

\begin{proof}
In both cases we can assume $X=\bigsqcup_{i}X_{i}$ with $f$ constant
on all $X_{i}$ and $X_{i}$ of finite type and non-empty. If $f$
is integrable then
\[
\{i\in I\st f(X_{i})=n\}\subseteq\{i\in I\st f(X_{i})+\dim X_{i}>n-1\}
\]
 is finite, that is $f^{-1}(n)$ is P-isomorphic to a scheme of finite
type. We can therefore assume $X=\bigsqcup_{n\in\frac{1}{l}\Z}X_{n}$
with $X_{n}=f^{-1}(n)$ (and allowing $X_{n}=\emptyset$). By \ref{lem:independence integrability}
integrability means that the sets $I_{m}=\{n\in\frac{1}{l}\Z\st n-\dim X_{-n}<m\}$
are finite. The limit in the statement means that all $I_{m}$ are
bounded above. Finally if $f$ is bounded above then all $I_{m}$
are bounded below. Conversely if $I_{0}$ is bounded below then $f$
is bounded above.
\end{proof}
\begin{rem}
\label{rem:realization of integrals}If we are given a continuous
ring homomorphism $\hat{\cM}_{k}^{\modified,l}\longrightarrow R$
of complete topological rings and continue to denote the image of
$\LL$ in $R$ by $\LL$, then we can similarly define integrals in
$R\cup\{\infty\}$. Of course, these integrals coincide with the images
of the corresponding integrals defined in $\hat{\cM}_{k}^{\modified,l}\cup\{\infty\}$.
\end{rem}

\section{Some results on power series rings\label{sec:Some-results-on-power-series}}

We collect in this section various results and notations about power
series rings. 
\begin{lem}
\label{lem:comp-tensor-localization-1}Let $R$ be a ring and $S$
be an $R$-algebra. Let $M$ be an $R[[t]]$-module and $M_{t}$ its
localization by $t$, which is an $R((t))=R[[t]]_{t}$-module. Then
we have
\[
(M\otimes_{R[[t]]}S[[t]])_{t}\simeq M_{t}\otimes_{R((t))}S((t)).
\]
\end{lem}

\begin{proof}
This follows from
\[
(M\otimes_{R[[t]]}S[[t]])\otimes_{S[[t]]}S[[t]]_{t}\simeq M_{t}\otimes_{R[[t]]_{t}}S[[t]]_{t}.
\]
\end{proof}
\begin{defn}
\label{def:complete-tensor}Let $R$ be a ring and $S$ be an $R$-algebra.
For an $R[[t]]$-module $M$, we define the \emph{complete tensor
product }as 
\[
M\hotimes_{R}S=M\otimes_{R[[t]]}S[[t]].
\]
\end{defn}

\begin{rem}
\label{lem:comp-tensor-localization} If $N$ is an $R((t))$-module
then by \ref{lem:comp-tensor-localization-1} we have
\[
N\hotimes_{R}S\simeq N\otimes_{R((t))}S((t)).
\]
In particular if $M$ is an $R[[t]]$-module and $S$ an $R$-algebra
then we have identifications
\[
(M\hotimes_{R}S)_{t}\simeq M_{t}\otimes_{R((t))}S((t))\simeq M_{t}\hotimes_{R}S.
\]
\end{rem}

\begin{lem}
\label{lem:complete-tensor} Let $R$ be a ring, $S$ be a Noetherian
$R$-algebra and $M$ be a finitely generated $R[[t]]$-module. Then
$M\otimes_{R}S\longrightarrow M\hotimes_{R}S$ is the completion with
respect to the ideal $(t)\subseteq R[[t]]$, that is we have a natural
isomorphism
\[
\varprojlim_{n\in\NN}\left((M/t^{n}M)\otimes_{R}S\right)\cong M\otimes_{R[[t]]}S[[t]].
\]
\end{lem}

\begin{proof}
The ring $S[[t]]$ is Noetherian and $t$-adically complete. Since
$M\otimes_{R[[t]]}S[[t]]$ is a finitely generated $S[[t]]$-module,
it is $t$-adically complete and the projective limit of
\[
N_{n}:=\left(M\otimes_{R[[t]]}S[[t]]\right)\otimes_{S[[t]]}\left(S[[t]]/(t^{n})\right)\quad(n\in\NN).
\]
Since 
\[
S[[t]]/(t^{n})\simeq S[[t]]\otimes_{R[[t]]}(R[[t]]/(t^{n}))\simeq R[[t]]/(t^{n})\otimes_{R}S,
\]
we have
\begin{align*}
N_{n} & \cong\left(M\otimes_{R[[t]]}S[[t]]\right)\otimes_{R[[t]]}\left(R[[t]]/(t^{n})\right)\\
 & \cong(M/t^{n}M)\otimes_{R[[t]]/(t^{n})}(S[[t]]/(t^{n}))\\
 & \cong(M/t^{n}M)\otimes_{R}S.
\end{align*}
The lemma follows. 
\end{proof}
\begin{rem}
\label{rem:complete tens is tens for finite algebras} By \cite[Lemma 2.4]{Tonini:2017qr}
if $S$ is a finite and finitely presented $R$-algebra then 
\[
\omega_{S/R}\colon R[[t]]\otimes_{R}S\longrightarrow S[[t]]
\]
is an isomorphism. In particular $M\hotimes_{R}S\simeq M\otimes_{R}S$
for all $R[[t]]$-modules $M$.
\end{rem}

\begin{lem}
\label{lem:uniformizable structure} Let $R$ be a ring, $k>0$ and
$g\in R[[s]]^{*}$. Then there exists a unique map $R[[t]]\arr R[[s]]$
of $R$-algebras mapping $t$ to $s^{k}g$ and $1,s,\dots,s^{k-1}$
is an $R[[t]]$-basis of $R[[s]]$. In particular if $R\arr R'$ is
a map of rings then $R[[s]]\hotimes_{R}R'(=R[[s]]\otimes_{R[[t]]}R'[[t]])\simeq R'[[s]]$
and $R((s))\hotimes_{R}R'\simeq R'((s))$.
\end{lem}

\begin{proof}
There are compatible maps $R[t]/(t^{n})\arr R[s]/(s^{n})$ mapping
$t$ to $s^{k}g$ and passing to the limit we get a map $R[[t]]\arr R[[s]]$.
Uniqueness is easy to prove. Consider the map
\[
\phi\colon R[[t]]^{k}\arr R[[s]]
\]
mapping the canonical basis to $1,s,\dots,s^{k-1}$. Notice that $R[[s]]/t^{n}R[[s]]\simeq R[s]/(s^{nk})$
because $g$ is invertible. Thus tensoring the above map by $R[[t]]/(t^{n})$
we obtain a map
\[
\phi_{n}\colon(R[t]/(t^{n}))^{k}\arr R[s]/(s^{nk}).
\]
In order to show that $\phi$ is an isomorphism it is enough to show
that all $\phi_{n}$ are isomorphisms. Since $\phi_{n}$ is a map
between free $R$-modules of the same rank, it is enough to show that
$\phi_{n}$ is surjective. By Nakayama's lemma we can assume $n=1$,
where the result is clear.

Using that $\phi$ is an isomorphism it is easy to conclude that the
map $R[[s]]\hotimes_{R}R'\arr R'[[s]]$ is an isomorphism. Since $t=s^{k}g$
we also have $R[[s]]_{t}=R((s))$, so that also the last isomorphism
holds.
\end{proof}
\begin{lem}
\label{lem:extension of iso on uniformizable up to nilpotents} Let
$R$ be a ring, $k>0$, $\zeta_{1},\zeta_{2}\in R[[s]]^{*}$ and consider
$R[[s_{i}]]$ as an $R[[t]]$ module via $R[[t]]\arr R[[s]]$, $t\longmapsto s^{k}\zeta_{i}$
for $i=1,2$. If $\sigma\colon R((s_{1}))\arr R((s_{2}))$ is an isomorphism
of $R((t))$-algebras then, up to modding out $R$ by finitely many
nilpotents, we have that $\sigma(R[[s_{1}]])=R[[s_{2}]]$, more precisely
there exists $u\in R[[s_{2}]]^{*}$ such that $\sigma(s_{1})=us_{2}$.
Moreover $\sigma_{|R[[s_{1}]]}:R[[s_{1}]]\arr R[[s_{2}]]$ is the
unique $R$-linear map sending $s_{1}$ to $us_{2}$. 
\end{lem}

\begin{proof}
From \ref{lem:uniformizable structure} we see that $R[[s_{i}]]$
is free of rank $k$ over $R[[t]]$, in particular $R[[t]]\subseteq R[[s_{i}]]$
is an integral extension. Notice moreover that $R[[s_{i}]]_{t}=R((s_{i}))$.
Set $\sigma(s_{1})=\sum_{m\in\Z}\sigma{}_{m}s_{2}^{m}\in R((s_{2}))$.
If $R$ is a field then $\sigma(s_{1})\in s_{2}R[[s_{2}]]^{*}$: $R[[s_{i}]]$
is a DVR with maximal ideal $(s_{i})$ and it is the integral closure
of $R[[t]]$ inside $R((s_{i}))$, so that
\[
\sigma_{|R[[s_{1}]]}\colon R[[s_{1}]]\arrdi{\simeq}R[[s_{2}]]\text{ and }(\sigma(s_{1}))=(s_{2})
\]
This means that all the $\sigma_{m}$ for $m\leq0$ lie in all the
prime ideals, that is they are nilpotent, and no prime ideal contains
$\sigma_{1}$, that is $\sigma_{1}$ is invertible. Modding out finitely
many nonzero $\sigma_{m}$ with $m\le0$ (there are at most finitely
many of them), we can assume that $\sigma(s_{1})=us_{2}$ as in the
statement. Since $R[[s_{1}]]$ is generated by $s_{1}$ as an $R[[t]]$
algebra it also follows that $\sigma(R[[s_{1}]])\subseteq R[[s_{2}]]$.
Doing the same for $\sigma^{-1}$ one also gets the equality. The
last statement follows from \ref{lem:uniformizable structure}.
\end{proof}
\begin{lem}
\label{lem:invariants are uniformazable} Let $R$ be a ring, $k>0$
and $\zeta\in R[[s]]^{*}$. Consider $R[[s]]$ as an $R[[t]]$ module
via $R[[t]]\arr R[[s]]$, $t\longmapsto s^{k}\zeta$ and assume that
$R((s))$ has a structure of $G$-torsor over $R((t))$, where $G$
is a finite group. Then $k=|G|$ and, up to modding out $R$ by finitely
many nilpotents, we have that:

\begin{enumerate}
\item for all $g\in G$ we have that $g(R[[s]])=R[[s]]$, more precisely
there exists $u_{g}\in R[[s]]^{*}$ such that $g(s)=u_{g}s$ and $g_{|R[[s]]}:R[[s]]\arr R[[s]]$
is the unique $R$-linear map sending $s$ to $u_{g}s$;
\item if $H<G$ is a subgroup and $s'=\prod_{h\in H}h(s)$ then $s'=s^{|H|}v$
with $v\in R[[s]]^{*}$, the map $R[[y]]\arr R[[s]]$, $y\longmapsto s'$
is an isomorphism onto $R[[s]]^{H}$, ($R[[s]]^{H})_{t}=(R[[s]]_{t})^{H}$
and $t=s'^{|G|/|H|}w$ with $w\in R[[s']]^{*}$. In particular $R[[s]]$
is a free $R[[s]]^{H}$-module of rank $|H|$ and $R[[s]]^{H}$ is
a free $R[[t]]$-module of rank $|G|/|H|$.
\end{enumerate}
\end{lem}

\begin{proof}
From \ref{lem:uniformizable structure} we see that $R[[s]]$ is free
of rank $k$ over $R[[t]]$. Since $R[[s]]_{t}/R((t))$ is a $G$-torsor
we can conclude that $k=|G|$. Point $(1)$ follows from \ref{lem:extension of iso on uniformizable up to nilpotents}.
Let us consider point $(2)$. We have $s'=s^{|H|}v$ where $v=\prod_{h}u_{h}\in R[[s]]^{*}$
for $u_{h}\in R[[s]]^{*}$ as in (1). By \ref{lem:uniformizable structure}
the map $\phi\colon R[[y]]\arr R[[s]]$, $\phi(y)=s'$, is well defined,
injective and $1,s,\dots,s^{|H|-1}$ is an $R[[y]]$-basis. Since
$s'\in R[[s]]^{H}$ it is easy to see that $\phi$ maps into $R[[s]]^{H}$.
Since $R[[s]]_{s'}=R((s))$ we have 
\[
R((y))\subseteq R((s))^{H}=(R[[s]]^{H})_{s'}\subseteq R((s)).
\]
Since $R((s))/R((s))^{H}$ is an $H$-torsor the $R((s))^{H}$-module
$R((s))$ is projective of rank $|H|$ and generated by $1,s,\dots,s^{|H|-1}$,
which is therefore a $R((s))^{H}$-basis: the induced map
\[
(R((s))^{H})^{|H|}\to R((s))
\]
is a surjective map of projective $R((s))^{H}$-modules of the same
rank, thus an isomorphism. Let $x\in R[[s]]^{H}\subseteq R[[s]]$
and write
\[
x=x_{0}+x_{1}s+\cdots+x_{|H|-1}s^{|H|-1}\text{ with }x_{i}\in R[[y]].
\]
Since we also have $x_{i}\in R((s))^{H}$ and the writing is unique
we conclude that $x_{1}=\cdots=x_{|H|-1}=0$ and $x=x_{0}$ in $R((s))^{H}$.
The injectivity of $R[[y]]\arr R((s))^{H}$ implies that $x\in R[[y]]$.
This shows that $R[[s']]=R[[s]]^{H}$ and 
\[
(R[[s]]^{H})_{t}=R((s'))=R((s))^{H}=(R[[s]]_{t})^{H}.
\]
Finally since $t\in R[[s]]^{H}$ we have $t=s'^{b}q$, where $q\in R[[s']]$,
$q(0)\neq0$. Thus $t=s^{b|H|}v^{b}q=s^{|G|}\zeta$. Looking at the
first non vanishing coefficient we conclude that $b|H|=|G|$ and that
$q$ is invertible.
\end{proof}
\begin{lem}
\label{lem:Zariski refined by uin} If $\{U_{i}\}_{i}$ is a Zariski
covering of $\Spec R((t))$ for some ring $R$ then there exists a
ubi covering $\{\Spec R_{j}\arr\Spec R\}_{j}$ such that $\Spec R_{j}((t))\arr\Spec R((t))$
factors through some of the $U_{i}$.
\end{lem}

\begin{proof}
We can assume $U_{i}=\Spec R((t))_{s_{i}}$ for $s_{1},\dots,s_{n}\in R[[t]]$
such that $(s_{1},\dots,s_{n})=R((t))$. This means there exist $a_{1},\dots,a_{n}\in R[[t]]$
and $r\in\N$ such that
\[
a_{1}s_{1}+\cdots+a_{n}s_{n}=t^{r}.
\]
If we write $s_{i}=\sum_{j}s_{i,j}t^{j}$ we can conclude that $(s_{i,j}\st j\leq r)=R$.
For the finite set $S$ of nonzero $s_{i,j}$ with $j\le r$, we let
$V_{s}=\Spec R_{s}$, then define $V_{J}^{\circ}=\Spec R_{J}$ for
each subset $J\subset S$. We can give scheme structures to $V_{J}^{\circ}$'s
such that the map $\coprod_{J}V_{J}^{\circ}\longrightarrow\Spec R$
is a ubi covering and for each $i$, $j$ and $J$, the element $s_{i,j}$
is either 0 or invertible in $R_{J}$. For each $J$, there exists
an index $i_{0}$ such that $s_{i_{0}}=t^{q}\omega$ with $\omega\in R_{J}[[t]]^{*}$.
In particular $s_{i_{0}}\in R_{J}((t))^{*}$. This implies that the
map $\Spec R_{J}((t))\longrightarrow\Spec R((t))$ factors through
$U_{i_{0}}$. 
\end{proof}

\section{Uniformization\label{sec:Uniformization}}

If $K$ is an algebraically closed field, then any finite étale $K((t))$-algebra
$A$ is $K$-isomorphic to a product of the power series field, $K((u))^{n}$,
for some $n\in\NN$, and its integer ring $\fO_{A}$  is isomorphic
to $K[[u]]^{n}$. This is no longer true if we replace $K$ with a
general ring. The goal of this section is to show that this however
becomes true after taking a sur covering of $\Spec K$.
\begin{defn}
\label{def:uniformizable}Let $R$ be a ring and $A$ be a finite
étale $R((t))$-algebra. We say that $A$ is \emph{uniformizable }(over
$R$) if there exist a finite decomposition $R\simeq\prod_{i=1}^{l}R_{i}$,
$n_{i}\in\N$ and isomorphisms $A\otimes_{R}R_{i}(=A\hotimes_{R}R_{i})\simeq R_{i}((s))^{n_{i}}$
such that each composition $R_{i}((t))\arr A\otimes_{R}R_{i}\arr R_{i}((s))$
maps $t$ to a series of the form $s^{k}g$ for some $k>0$ and $g\in R_{i}[[s]]^{*}$.
\end{defn}

\begin{rem}
\label{rem:uniformizable extends to a flat cover} If $A/R((t))$
is an uniformizable finite étale $R((t))$-algebra and we use notation
from \ref{def:uniformizable} then $\fO=\prod_{i}R_{i}[[s]]^{n_{i}}$
is a finite and flat $R[[t]]$ algebra with an isomorphism $\fO_{t}\simeq A$.
It is not clear if a general finite étale $R((t))$-algebra always
admits a finite and flat extension on $R[[t]]$, not even fpqc locally.
\end{rem}

\begin{thm}[Uniformization]
\label{cor:stratified-uniformization} Let $R$ be a ring and $A$
be a finite and étale $R((t))$-algebra. Then there exists a surjective
and finitely presented map $\Spec S\arr\Spec R$ such that $A\hotimes_{R}S$
is uniformizable. In other words $A$ is uniformizable sur locally
in $R$.
\end{thm}

\begin{proof}
Let $S_{0}$ be the henselization of $R[t]$ with respect to the ideal
$(t)$. From \cite[Th. 7 and pages 588-589]{MR0345966} or \cite[Th. 5.4.53]{MR2004652},
there exists a finite étale cover $S_{0}[t^{-1}]\longrightarrow A_{0}$
such that $A_{0}\otimes_{S_{0}[t^{-1}]}R((t))\cong A$. In turn there
exist an étale neighborhood $R[t]\longrightarrow S_{1}$ of $(t)$,
that is with $R\simeq S_{1}/tS_{1}$, and a finite étale cover $S_{1}[t^{-1}]\longrightarrow A_{1}$
such that $A_{1}\otimes_{S_{1}[t^{-1}]}S_{0}[t^{-1}]\cong A_{0}$
and $A_{1}\otimes_{S_{1}[t^{-1}]}R((t))\cong A$. Since $S_{1}$ and
$A_{1}$ are finitely generated over $R$, there exist a finitely
generated subalgebra $R'\subset R$, an étale neighborhood $R'[t]\longrightarrow S_{2}$
and a finite étale cover $S_{2}[t^{-1}]\longrightarrow A_{2}$ which
induce $R[t]\longrightarrow S_{1}$ and $S_{1}[t^{-1}]\longrightarrow A_{1}$
by the scalar extension $R/R'$. Then $A\cong A_{2}\otimes_{S_{2}[t^{-1}]}R((t))$.
If we put $A'=A_{2}\otimes_{S_{2}[t^{-1}]}R'((t))$, then $A\cong A'\otimes_{R'((t))}R((t))$.
Therefore it suffices to show that $R'((t))\longrightarrow A'$ is
sur locally uniformizable.

We claim that there exist a sur covering $\coprod_{i}\Spec R_{i}\longrightarrow\Spec R'$
and a commutative diagram for each $i$,
\[
\xymatrix{A_{2}\otimes_{R'}R_{i}\ar[dr]\ar[rr] &  & \Spec Q_{i}\ar[dl]\\
 & \Spec S_{2}\otimes_{R'}R_{i}
}
\]
such that 
\begin{enumerate}
\item $R_{i}$ is a domain,
\item we have $(Q_{i})_{t}\simeq A_{2}\otimes_{R'}R_{i}$,
\item the lower left arrow is the one induced from $S_{2}\longrightarrow S_{2}[t^{-1}]\longrightarrow A_{2}$,
\item the lower right arrow is a finite morphism,
\item each connected component of $(\Spec Q_{i}/tQ_{i})_{\red}$ maps isomorphically
onto $\Spec R_{i}=\Spec S_{2}\otimes_{R'}R_{i}/t(S_{2}\otimes_{R'}R_{i})$,
\item up to shrink $S_{2}$ to a smaller neighborhood of $(t)$, $\sqrt{tQ_{i}}$
is a principal ideal generated by some $q_{i}\in Q_{i}$.
\end{enumerate}
Let's see how to conclude from this. By \cite[Corollary 7.5]{MR1322960},
the ring $Q_{i}\otimes R_{i}[[t]]$ is a product of rings $P_{1}\times\cdots\times P_{l}$
such that the reduction of $P_{j}/tP_{j}$ is $R_{i}$. The map $R[[x]]\arr P_{i}$,
$x\arr q_{i}$ is well defined and surjective because $P_{i}$ is
$t$-adically and therefore $q_{i}$-adically complete. Since $\dim R[[x]]=\dim P_{i}$
and $R[[x]]$ is a domain the map $R[[x]]\arr P_{i}$ is an isomorphism.
In conclusion $Q_{i}\otimes R_{i}[[t]]\simeq R_{i}[[q_{i}]]^{n_{i}}$
for some $n_{i}\in\N$.

This implies that $A_{2}\otimes_{S_{2}[t^{-1}]}R_{i}((t))\cong A'\otimes_{R'((t))}R_{i}((t))\cong R_{i}((q_{i}))^{n_{i}}$.
The image of $t$ in each factor $R_{i}[[q_{i}]]$ is of the form
$q_{i}^{k}g$ for some $k>0$ and $g\in R_{i}[[q_{i}]]\setminus q_{i}R_{i}[[q_{i}]]$.
Inverting the constant term of $g$ for each factor, we get an open
dense subscheme $\Spec R_{i}'\subset\Spec R_{i}$ such that $R_{i}'((t))\longrightarrow A'\otimes_{R'((t))}R_{i}'((t))$
is uniformizable. By Noetherian induction, we conclude that $R_{i}((t))\longrightarrow A'\otimes_{R'((t))}R_{i}((t))$
is sur locally uniformizable. Therefore $R'((t))\longrightarrow A'$
is also sur locally uniformizable and the theorem follows.

It remains to prove the claim. Note that $R'$ is finitely generated
over $\Z$, in particular, a Noetherian ring of finite dimension.
By \ref{lem:decomposing something locally of finite type}, we may
assume that $R'$ is a domain and it is enough to show that there
exists one dominant finite-type morphism $\Spec R_{i}\longrightarrow\Spec R'$
satisfying the above conditions. Let $K$ be an algebraic closure
of the fraction field $K'$ of $R'$. The map $\Spec A_{2}\otimes_{R'}K\longrightarrow\Spec S_{2}[t^{-1}]\otimes_{R'}K$
is an étale finite cover of affine algebraic curves over $K$. Taking
a partial compactification of $\Spec A_{2}\otimes_{R'}K$, we can
extend this cover to a finite (not necessarily étale) cover $\Spec Q_{K}\longrightarrow\Spec S_{2}\otimes_{R'}K$
with $\Spec Q_{K}$ smooth. Let $p_{1},\dots,p_{m}\colon\Spec K\longrightarrow\Spec Q_{K}$
be the points lying over the point $\Spec K=V(t)\hookrightarrow\Spec S_{2}\otimes_{R'}K$.
We take a sufficiently large intermediate field $L$ between $K$
and $K'$ which is finite over $K'$ and such that $Q_{K}$ and morphisms
$\Spec Q_{K}\longrightarrow\Spec S_{2}\otimes_{R'}K$ and $p_{i}$
are all defined over $L$ (see \cite[Tag 01ZM]{stacks-project}, \cite[Tag 01ZN]{stacks-project}).
Denote by $Q_{L}\to\Spec S_{2}\otimes_{R'}L$ the obtained map. As
$Q_{L}\otimes_{L}K\simeq Q_{K}$ we can conclude that $Q_{L}$ is
smooth over $L$. Replacing $R'$ with its integral closure in $L$,
we can assume $K'=L$. Since $\Spec Q_{K'}$ is normal, we can extend
$\Spec Q_{K'}\to\Spec(S_{2}\otimes_{R'}K')$ to $\Spec S_{2}$ by
taking the normalization $\Spec Q\to\Spec S_{2}$. As $\Spec Q_{t}\to\Spec S_{2}[t^{-1}]$
and $\Spec A_{2}\to\Spec S_{2}[t^{-1}]$ are isomorphic over $K$,
shrinking $\Spec R'$ we can assume that $Q_{t}\simeq A_{2}$ (see
\cite[Tag 081E]{stacks-project}). So far we have proved that $\Spec Q$
satisfies points $(1)$ to $(4)$.

For $(5)$, consider $\Spec(Q/tQ)_{\text{\ensuremath{\red}}}\to\Spec(S_{2}/tS_{2})=\Spec R'$.
We have
\[
(Q/tQ)_{\red}\otimes_{R'}K'\simeq(Q_{K'}/tQ_{K'})_{\red}\simeq K'^{m}
\]
because $(Q_{K'}/tQ_{K'})$ is a finite $K'$ algebra with only rational
points. Thus, shrinking $R'$, we can assume $(Q/tQ)_{\red}\simeq R'^{m}$,
that is point $(5)$.

We now focus on $(6)$. Since $\Spec Q\to\Spec R'$ is generically
smooth, by \cite[Tag 0C0C]{stacks-project} we can assume it is smooth.
Since $R'$ is excellent, we can also assume that $R'$ is regular,
so that $Q$ is regular as well. Let $J=\sqrt{tQ}$. If $x=tS_{2}\in\Spec S_{2}$
is the unique point over $0=(t)\in\Spec R'[t]$, we have that that
$Q_{x}$ is a finite extension of $(S_{2})_{x}\simeq R[t]_{(t)}$.
In particular $Q_{x}\otimes_{R'}K'$ is a semilocal regular ring of
dimension $1$, thus a finite product of semilocal Dedekind domains.
By \cite[Chapter 13, Corollary 1.4]{Auslander2014} it is therefore
a PID, so that $(J_{x})\otimes_{R'}K'\simeq Q_{x}\otimes_{R'}K'$.
By \cite[Tag 01ZM]{stacks-project} we can assume it extends to an
isomorphism $J_{x}\simeq Q_{x}$. Shrinking $S_{2}$ étale locally
around $x$, we can finally assume $J\simeq Q$, which yields condition
$(6)$.
\end{proof}

\section{The P-moduli space of formal torsors\label{sec:The-P-moduli-space-of-Delta}}

Let $k$ be a base field and $G$ be a finite group. We prove the
existence of P-moduli space of torsors over $k((t))$ for a fixed
finite group $G$ or the one of finite étale covers of $k((t))$ of
fixed degree.
\begin{notation}
\label{nota:semidirect characteristic} We set $p=\car k$, allowing
it also to be $0$. In this section, with abuse of notation, we often
consider groups of the form $H\rtimes C$ where $H$ is a $p$-group
and $C$ is a tame cyclic group. If $p>0$ this means that $C$ is
a cyclic group whose order is coprime with $p$. If $p=0$ instead
this means that $H=0$, while $C$ is any cyclic group.
\end{notation}

\begin{defn}
The functor $\Delta_{n}\colon(\Aff/k)^{\op}\arr\sets$ maps a ring
$R$ to the set of isomorphism classes of finite étale covers of $R((t))$
of constant degree $n$. For a morphism $f\colon\Spec S\longrightarrow\Spec R$
of affine $k$-schemes the \emph{pull-back }map $f^{*}:\Delta_{n}(R)\longrightarrow\Delta_{n}(S)$
sends an étale $R((t))$-algebra $A$ to $A\hotimes_{R}S=A\otimes_{R((t))}S((t))$. 

We define a functor $\Delta_{G}\colon(\Aff/k)^{\op}\arr\sets$ mapping
a ring $R$ to the set of isomorphism classes of $G$-tosors over
$R((t))$. The pullback is defined similarly to the one of $\Delta_{n}$.
\end{defn}

Since finite étale algebras correspond to $S_{n}$-torsors by \cite[Prop. 1.6]{Tonini:aa},
we have an isomorphism $\Delta_{n}\simeq\Delta_{S_{n}}$.
\begin{rem}
Notation above slightly differs to the notation used in \cite{Tonini:2017qr},
where $\Delta_{G}$ denotes the analogous fiber category.
\end{rem}

\begin{lem}
\label{lem:semidirec-case}Suppose that $G$ is the semidirect product
$H\rtimes C$ of a $p$-group $H$ and a tame cyclic group $C$. Then
the functor $\Delta_{G}$ has a strong P-moduli space which is the
disjoint union of countably many affine schemes of finite type over
$k$.
\end{lem}

\begin{proof}
This follows from \cite[Theorem A]{Tonini:2017qr} and Lemma \ref{lem:strong moduli of ind coarse}.
\end{proof}
\begin{defn}
For $\cF\colon(\Aff/k)^{\op}\arr\sets$ and $A\in\cF(V)$, a \emph{geometric
fiber }of $A$ is the image of $A$ under the map $\cF(V)\longrightarrow\cF(K)$
associated with some geometric point $\Spec K\longrightarrow V$.
\end{defn}

\begin{defn}
We denote by $\Delta_{n}^{\circ}$ (resp. $\Delta_{G}^{\circ}$) the
subfunctor of $\Delta_{n}$ (resp. $\Delta_{G}$) consisting of étale
$R((t))$-algebras $A$ whose geometric fibers are connected: for
every algebraically closed $R$-field $K$, the induced algebra $A\hotimes_{R}K$
is a field. 
\end{defn}

\begin{comment}
%
\begin{defn}
By putting the superscript ``$\unif$'' on one of the functors defined
above, we mean the subfunctor of uniformizable algebras. For instance,
$\Delta_{G}^{\unif,\circ}(R)$ is the set of isomorphism classes of
uniformizable $G$-Galois $R((t))$-algebras with connected geometric
fibers. Those are subfunctors of the original one thanks to \ref{lem:uniformizable structure}.
\end{defn}

\begin{defn}
We define $\Phi_{G}$ as the subfunctor of $\Delta_{G}$ of torsors
which are sur locally uniformizable.
\end{defn}

\begin{rem}
By \ref{cor:stratified-uniformization} we have that $\Phi_{G}=\Delta_{G}$
if $G$ is solvable and, in general, that $\Phi_{G}(R)=\Delta_{G}(R)$
if $R$ is a Noetherian ring. Notice moreover that by definition $\Delta_{G}^{\unif}\arr\Phi_{G}$
is geometrically bijective and an epimorphism in the sur topology,
so that $(\Delta_{G}^{\unif})^{\P}\simeq\Phi_{G}^{\P}$.
\end{rem}

\begin{rem}
The isomorphism $\Delta_{n}\arr\Delta_{S_{n}}$ does not preserves
the $(-)^{\circ}$ part. Moreover it is unclear if $\Delta_{n}^{\unif}$
maps into $\Phi_{S_{n}}$: by \todo{Reference to Pro 1.6, ESSENTIALLY FINITE VECTOR BUNDLES ON NORMAL PSEUDO-PROPER ALGEBRAIC STACKS}
this boils down to prove that a tensor product of uniformazable algebras
is uniformazable.
\end{rem}

\end{comment}

\begin{lem}
\label{lem:geom connected is constructible} The property ``being
connected'' is locally constructible for both $\Delta_{G}$ and $\Delta_{n}$.
\end{lem}

\begin{proof}
Let $\shF=\Delta_{G}$ or $\shF=\Delta_{n}$ and let $\shQ\subseteq|\shF|$
denote the property in the statement. Given $\psi\colon V\to\shF$
we have to show that $|\psi|^{-1}(\shQ)=\shQ_{V}\subseteq V$ is locally
constructible. By \ref{lem:locally constructible for F and FP} the
problem is sur local in $V$. By Theorem \ref{cor:stratified-uniformization}
we can therefore assume that $V=\Spec R$ and that $\psi$ corresponds
to a torsor/étale map $R((t))\to A=R((s))^{m}$ with $t=s^{k}g$,
$g\in R[[s]]^{*}$ and $k,m>0$. The subset $\shQ_{V}$ is locally
constructible in $V$ because, if $m=1$ then $\shQ_{V}=V$, while
if $m>1$ then $\shQ_{V}=\emptyset$.
\end{proof}
\begin{lem}
\label{lem:torsors and disjoint union} Let $f\colon X\arr S$ be
a $G$-torsor over a scheme $S$ and let $X=\bigsqcup_{i}U_{i}$ be
a finite decomposition into open subsets. Then:
\begin{itemize}
\item there exist finite decompositions into open subsets $S=\bigsqcup_{j}S_{j}$
and $f^{-1}(S_{j})=\bigsqcup_{k}V_{jk}$ with the following properties:
for all $j,k$ there exists $i$ such that $V_{jk}\subseteq U_{i}$;
the group $G$ permutes the $V_{jk}$ and, for all $j$, $G$ acts
transitively on $\{V_{jk}\}_{k}$;
\item if $G$ permutes transitively the $U_{i}$ and $G_{i}$ is the stabilizer
of $U_{i}$ in $G$ then $U_{i}$ is an $G_{i}$-torsor over $S$.
\end{itemize}
\end{lem}

\begin{proof}
Let $T=\{g(U_{i})\}_{i,g\in G}$ and, for $J\subseteq T$, set 
\[
T_{J}=(\bigcap_{V\in J}V)\cap(\bigcap_{V\notin J}(X-V))
\]
so that $X$ is the disjoint union of the $T_{J}$ and all $U_{i}$
are a disjoint union of some of the $T_{J}$. Notice that $G$ permutes
the $T_{J}$: given $g\in G$ and $J\subseteq T$ one has that $g(T_{J})=T_{g(J)}$
where $g(J)=\{g(V)\st V\in J\}$. This also implies that $f(T_{J})\cap f(T_{J'})\neq\emptyset$
only if there exists $g\in G$ such that $g(T_{J})=T_{J'}$, in which
case $f(T_{J})=f(T_{J'})$: if $s=f(x)=f(y)$ with $x\in T_{J}$ and
$y\in T_{J'}$ then there exists $g\in G$ such that $g(x)=y\in g(T_{J})\cap T_{J'}$,
so that $g(T_{J})=T_{J'}$. Up to removing repetitions the sets $f(T_{J})$
yield the desired decomposition of $S$.

We now consider the last statement. If $S'\to S$ is any map with
$S'\neq\emptyset$ denote by $\psi\colon X'\to X$ its base change
along $f\colon X\to S$. The collection $\{\psi^{-1}(U_{i})\}_{i}$
defines a partition of $X'$ over which $G$ acts transitively. In
particular $\psi^{-1}(U_{i})\neq\emptyset$ for all $i$. We claim
that $G_{i}$ is also the stabilizer of $\psi^{-1}(U_{i})$ in $G$.
Indeed if $g(\psi^{-1}(U_{i}))=\psi^{-1}(g(U_{i}))=\psi^{-1}(U_{i})$
then $g(U_{i})\cap U_{i}\neq\emptyset$ and therefore $g(U_{i})=U_{i}$.
In particular we can assume $X=G\times S$ and, by transitivity, that
$\{1\}\times S\in U_{i}$, so that $G_{i}\times S\subseteq U_{i}$.
In order to prove that this last inclusion is an equality, it suffices
to show the corresponding equality over each point of $S$ and thus
we can further assume $S=\Spec K$, for some field $K$. In this case
\[
g\in U_{i}\then g=g\cdot1\in g(U_{i})\cap U_{i}\neq\emptyset\then g(U_{i})=U_{i}\then g\in G_{i}
\]
\end{proof}
\begin{thm}
\label{thm:strong P moduli of DeltaG} Let $k$ be a field, $G$ a
finite group and let $\mathcal{Q}$ be a locally constructible property
for $\Delta_{G}$. Then $\Delta_{G}^{\cQ}$ (e.g. $\Delta_{G}$ or
$\Delta_{G}^{\circ}$) has a strong P-moduli space which is a countable
disjoint union of affine $k$-varieties.
\end{thm}

\begin{proof}
By \ref{prop:submoduli-const} and \ref{lem:decomposing something locally of finite type}
it is enough to consider the case of $\Delta_{G}$.

In the semidirect case $G=H\rtimes C$, for a $p$-group $H$ and
a tame cyclic group $C$, the claim follows from \ref{lem:semidirec-case}.
In this case we denote the P-moduli space of $\Delta_{G}^{\circ}$
by $\overline{\Delta_{G}^{\circ}}$. %

Let now $G$ be an arbitrary finite group. Let $\Lambda$ be the set
of representatives of $G$-conjugacy classes of subgroups $H\subset G$
which are isomorphic to the semi-direct product $B\rtimes C$ of a
$p$-group $B$ and a tame cyclic group $C$. Let also $\Aut_{G}(H)$
denote the subgroup of automorphisms of $H$ induced by conjugation
of an element of $G$. There exists a natural action of $\Aut_{G}(H)$
on $\Delta_{H}^{\circ}$, inducing a natural P-action on $\overline{\Delta_{H}^{\circ}}$.
From \ref{prop:quotient by finite group }, there exists the strong
P-quotient $\overline{\Delta_{H}^{\circ}}/\Aut_{G}(H)$, which is
a strong P-moduli space of the quotient functor $\Delta_{H}^{\circ}/\Aut_{G}(H)$
and it is P-isomorphic to a disjoint union of $k$-varieties.

Consider the map $\ind_{H}^{G}\colon\Delta_{H}^{\circ}\longrightarrow\Delta_{G}$.
This is $\Aut_{G}(H)$-invariant and induces maps $\Delta_{H}^{\circ}/\Aut_{G}(H)\longrightarrow\Delta_{G}$
and
\[
\bigsqcup_{H\in\Lambda}\Delta_{H}^{\circ}/\Aut_{G}(H)\longrightarrow\Delta_{G}.
\]
We claim that this is geometrically injective and an epimorphism in
the sur topology. Since the source of this map has a strong P-moduli
space as in the theorem, from \ref{lem:when phiP is an isomorphism},
the claim implies the theorem. It remains to show the claim. 

\emph{Epimorphism. }Follows from \ref{cor:stratified-uniformization}
and \ref{lem:torsors and disjoint union} and the fact that Galois
extensions of $K((t))$ with $K$ algebraically closed has Galois
group a semidirect product of a $p$-group and a cyclic tame group.

\emph{Geometrically injective. }If $K$ is an algebraically closed
field then $\Delta_{H}^{\circ}(K)/\Aut_{G}(H)$ is the set of isomorphism
classes of Galois extensions $L/K((t))$ modded out by the equivalence
relation induced by the action of $\Aut_{G}(H)$. Given such an object
the corresponding $G$-torsor is $\ind_{H}^{G}L$. Let $L\in\Delta_{H}^{\circ}(K)$
and $L'\in\Delta_{H'}^{\circ}(K)$ for $H,H'\in\Lambda$ be such that
$\ind_{H}^{G}L\simeq\ind_{H'}^{G}L'$ as $G$-torsors. It follows
that $L'$ is one of the component of $\ind_{H}^{G}L$ and $H'$ is
its stabilizer. Thus $H'=gHg^{-1}$ for $g\in G$ and $L'=L$ with
the $H'$ action induced by $H$. But since $\Lambda$ is a set of
representative we obtain $H'=H$ and therefore $L,L'\in\Delta_{H}^{\circ}(K)$
are in the same orbit for the action of $\Aut_{G}(H)$.
\end{proof}
\begin{cor}
\label{cor:P-module-etale-gp-scheme}Theorem \ref{thm:strong P moduli of DeltaG}
holds also when $G$ is a finite étale group scheme over $k$. 
\end{cor}

\begin{proof}
Let $k'/k$ be a finite Galois extension with Galois group $H$ such
that $G\otimes_{k}k'$ is a constant group. After the base change
to $k'$, the functor $\Delta_{G}^{\cQ}$ has a strong P-moduli space.
This space has a P-action of $H$. It suffices to take the strong
P-quotient, which exists from \ref{prop:quotient by finite group }.
\end{proof}

\section{Local constructibility of weighting functions\label{sec:Local-constructibility}}

In the wild McKay correspondence, there appear motivic integrals of
the form $\int_{\Delta_{G}}\LL^{f}$ for some weighting functions
$f\colon\Delta_{G}\longrightarrow\frac{1}{|G|}\ZZ$. In this section,
we show that these functions $f$ are locally constructible, which
proves that these integrals indeed make sense. 

We first recall the definitions of these functions. We fix a free
$k[[t]]$-module $M=k[[t]]^{\oplus r}$ of rank $r$ endowed with
a $k[[t]]$-linear $G$-action. For a $k$-algebra $B$, we let $M_{B}:=M\hat{\otimes}_{k}B$. 
\begin{defn}
For a field extension $K/k$ and a $G$-torsor $A/K((t))$, we define
a number $v_{M}(A)\in\frac{1}{|G|}\ZZ$ by 
\[
v_{M}(A):=\frac{1}{|G|}l_{K[[t]]}\left(\frac{\Hom_{k[[t]]}(M,\cO_{A})}{\cO_{A}\cdot\Hom_{k[[t]]}^{G}(M,\cO_{A})}\right),
\]
where $\Hom_{k[[t]]}^{G}(M,\fO_{A})$ is the set of $G$-equivariant
$k[[t]]$-linear maps, $\odi A$ is the integral closure of $K[[t]]$
inside $A$ and $l_{K[[t]]}$ denotes the length of a $K[[t]]$-module.
The natural map 
\[
M_{K}\longrightarrow E_{K}:=\Hom_{K[[t]]}(\Hom_{K[[t]]}^{G}(M_{K},\cO_{A}),\cO_{A}),\,f\longmapsto(\psi\longmapsto\psi(f))
\]
induces a map 
\[
\eta_{A}\colon\Spec S_{\cO_{A}}^{\bullet}E_{K}\longrightarrow\Spec S_{K[[t]]}^{\bullet}M_{K},
\]
where $S_{R}^{\bullet}M$ denotes the symmetric algebra of an $R$-module
$M$. Let 
\[
o\in(\Spec S_{K[[t]]}^{\bullet}M_{K})(K)
\]
be the $K$-point at the origin. We define 
\[
w_{M}(A):=\dim\eta_{A}^{-1}(o)-v_{M}(A)\in\frac{1}{|G|}\ZZ.
\]
Using the above functions we define maps 
\[
v_{M},w_{M}\colon\Delta_{G}\arr\frac{1}{|G|}\ZZ.
\]
We will just write $v=v_{M}$ and $w=w_{M}$ when this creates no
confusion.
\end{defn}

\begin{rem}
For slight differences of definitions of $v$ and $w$ appearing in
the literature, see \cite[Rem. 8.2]{MR3730512} and errata to the
paper \cite{MR3665638}, available online.\footnote{https://msp.org/ant/2017/11-4/p02.xhtml}
\end{rem}

We are going to prove that the maps $v,w\colon\Delta_{G}\arr\frac{1}{|G|}\ZZ$
are well defined and locally constructible. If we set $N=\Hom_{k[[t]]}(M,k[[t]])$
and think of it as a $k[[t]]$-module with an action of $G$ we have
isomorphisms 
\[
N\otimes_{k[[t]]}\odi A\simeq\Hom_{k[[t]]}(M,\cO_{A})\simeq N_{K}\otimes_{K[[t]]}\odi A\simeq\Hom_{K[[t]]}(M_{K},\cO_{A}).
\]

\begin{lem}
\label{lem:weight function module property} The $K[[t]]$-module
$\Hom_{k[[t]]}^{G}(M,\cO_{A})$ is free of rank $r$ and the map 
\[
\Hom_{k[[t]]}^{G}(M,\cO_{A})\otimes_{K[[t]]}\odi A\arr\cO_{A}\cdot\Hom_{k[[t]]}^{G}(M,\cO_{A})\subseteq\Hom_{k[[t]]}(M,\odi A)
\]
is an isomorphism. In particular the number $v(A)$ is well defined,
that is finite. 
\end{lem}

\begin{proof}
We can assume $K=k$. The module $\Hom_{k[[t]]}^{G}(M,\cO_{A})$ is
contained in a free $k[[t]]$-module and therefore it is free. In
order to compute its rank and prove that the map in the statement
is injective we can check what happens after localizing by $t$. If
we set $R=k((t))$ and $N_{t}=Q$ we have that the map $(Q\otimes_{R}A)^{G}\otimes_{R}A\arr Q\otimes_{R}A$,
fppf locally on $R$ after trivializing $A$, become
\[
(Q\otimes_{R}R[G])^{G}\otimes_{R}R[G]\arr Q\otimes_{R}R[G].
\]
In particular $(Q\otimes_{R}R[G])^{G}\simeq Q$, the corresponding
map $Q\arr Q\otimes_{R}R[G]$ is the coaction and the above map is
an isomorphism.
\end{proof}
\begin{rem}
The function $v$ is equal to the $t$-order of the ideal $\mathfrak{r}^{*}(L|\mathfrak{o},M)\subset\cO_{A}$,
Fröhlich's module resolvent \cite[Sec. 3]{MR0414520}. This follows
from determinantal descriptions of both values (see \cite[Def. 6.5]{MR3791224},
\cite[Def. 3.3]{MR3431631} and \cite[Sec. 3]{MR0414520}).
\end{rem}

\begin{lem}
Let $K'/K/k$ be field extensions, $A/K((t))$ a $G$-torsor and $A_{K'}=A\hat{\otimes}_{K}K'$
the associated $G$-torsor over $K'((t))$. Then $v(A)=v(A_{K'})$
and $w(A)=w(A_{K'})$. In particular the maps $v,w\colon\Delta_{G}\arr\frac{1}{|G|}\ZZ$
are well defined.
\end{lem}

\begin{proof}
Recall that the operations of taking invariants and flat base change
commute. From \ref{lem:weight function module property} and the fact
that 
\[
l_{K'[[t]]}(Q\hotimes_{K}K')=l_{K[[t]]}(Q)\text{ for }Q\in\Mod(K[[t]])
\]
one gets $v(A)=v(A_{K'})$. Similarly one obtains that $\eta_{A_{K'}}^{-1}(o)\simeq\eta_{A}^{-1}(o)\times_{K}K'$
and therefore the equality for the dimensions.
\end{proof}
\begin{lem}[{cf. \cite[Lem. 3.4]{MR3431631}}]
\label{lem:induction and local constructibility} Let $H$ be a subgroup
of $G$. Then we have the equalities of functions
\[
v_{\Res_{H}M}=v_{M}\circ\ind_{H}^{G},w_{\Res_{H}M}=w_{M}\circ\ind_{H}^{G}\colon\Delta_{H}\arr\Delta_{G}\arr\frac{1}{|G|}\ZZ.
\]
\end{lem}

\begin{proof}
Assume a $G$-torsor $A/K((t))$ is induced by an $H$-torsor $B/K((t))$.
We have isomorphisms
\[
A\simeq\ind_{H}^{G}B\simeq B^{|G/H|}\comma\odi A\simeq\ind_{H}^{G}\odi B\simeq\odi B^{|G/H|}
\]
and $\Hom_{k[[t]]}^{G}(M,\odi A)\simeq\Hom_{k[[t]]}^{H}(M,\odi B)$.
Moreover
\[
\odi A\Hom_{k[[t]]}^{G}(M,\odi A)\simeq(\odi B\Hom_{k[[t]]}^{H}(M,\odi B))^{|G/H|}
\]
inside $(\Hom_{k[[t]]}(M,\odi B))^{|G/H|}\simeq\Hom_{k[[t]]}(M,\odi A)$.
This proves that $v(A)=v(B)$. Finally one can check that
\[
\eta_{A}\colon\AA_{\cO_{A}}^{r}=\A_{\odi B}^{r}\amalg\cdots\amalg\A_{\odi B}^{r}\arr\A_{K[[t]]}^{r}
\]
 and that all maps $\A_{\odi B}^{r}\arr\A_{K[[t]]}^{r}$ are isomorphic
to $\eta_{B}$. It follows that $\dim\eta_{A}^{-1}(o)=\dim\eta_{B}^{-1}(o)$.
\end{proof}
To show properties of $v$ and $w$, we give slightly different descriptions
of these functions. For simplicity assume that $A$ is uniformizable
and connected, that is a Galois extension of $K((t))$ with group
$G$ and $\odi A=K[[s]]$. Let 
\[
a_{j}={}^{t}(a_{1j},\dots,a_{rj})\in N\otimes_{k[[t]]}\cO_{A}\simeq\cO_{A}^{r}\quad(j=1,\dots,r)
\]
be a $K[[t]]$-basis of $(N\otimes_{k[[t]]}\odi A)^{G}$ , which,
by \ref{lem:weight function module property}, is also an $\odi A$-basis
of $\odi A(N\otimes_{k[[t]]}\odi A)^{G}$. Since $l_{K[[t]]}=l_{K[[s]]}$
and using standard properties of DVR's we get
\begin{equation}
v(A)=\frac{\ord_{A}\det(a_{ij})}{|G|},\label{eq:v(A)}
\end{equation}
where $\ord_{A}$ denotes the normalized additive valuation on $A$,
that is, the order in $s$.

Let $e_{1},\dots,e_{r}$ be the standard basis of $M$ and $b_{1},\dots,b_{r}\in E_{K}$
the dual basis of $a_{1},\dots,a_{r}$. The map $M_{K}\longrightarrow E_{K}$
sends $e_{i}$ to $\sum_{j}a_{ij}b_{j}$. If $\Omega$ is the residue
field of $\odi A$ then
\[
(\eta_{A}^{-1}(o))_{\red}=\Spec(\frac{\Omega[X,\dots,X_{r}]}{(\sum_{j}\overline{a_{ij}}X_{j})})
\]
where $\overline{a_{ij}}$ is the image of $a_{ij}\in\odi A$ in $\Omega$.
It follows that 
\begin{equation}
\dim\eta_{A}^{-1}(o)=r-\rank(\overline{a_{ij}}).\label{eq:w(A)}
\end{equation}

\begin{lem}
\label{lem:consistent basis}Let $A/B((t))$ be a $G$-torsor such
that $A=B((s))$ is uniformizable and $\odi A=B[[s]]$ is $G$ invariant.
Assume moreover that $B$ is a Noetherian ring. Then there exist a
sur covering $\Spec B'\longrightarrow\Spec B$ such that $(N\otimes_{k[[t]]}(\odi A\hotimes_{B}B'))^{G}$
is a free $B'[[t]]$-module of rank $r$ and, for any ring map $B'\longrightarrow C$,
the base change map
\[
[N\otimes_{k[[t]]}(\odi A\hotimes_{B}B')]^{G}\hotimes_{B'}C\arr[N\otimes_{k[[t]]}(\odi A\hotimes_{B}C)]^{G}
\]
is an isomorphism.
\end{lem}

\begin{proof}
By \ref{rem:Noetherian induction} we may suppose that $B$ is a domain
and show that there exists an affine open dense subscheme $\Spec B'\hookrightarrow\Spec B$
satisfying the requests of the lemma. Given a $B$-algebra $C$ we
set $A_{C}=A\hotimes_{B}C$, $\odi{A_{C}}=\odi A\hotimes_{B}C\simeq C[[s]]$
and
\[
\phi_{A_{C}}\colon N\otimes_{k[[t]]}\odi{A_{C}}\arr\bigoplus_{g\in G}N\otimes_{k[[t]]}\odi{A_{C}},\alpha\longmapsto(g\alpha-\alpha)_{g}
\]
so that $(N\otimes_{k[[t]]}\odi{A_{C}})^{G}=\Ker\phi_{A_{C}}$. Let
$S$ be the localization of $B[[t]]$ at the prime ideal $(t)$, which
is a discrete valuation ring. Let us consider the map
\[
\phi_{A}\otimes_{B[[t]]}S\colon N\otimes_{k[[t]]}\fO_{A}\otimes_{B[[t]]}S\longrightarrow\bigoplus_{g\in G}N\otimes_{k[[t]]}\fO_{A}\otimes_{B[[t]]}S.
\]
From \cite[VII. 21]{MR643362}, there exist $S$-bases $\alpha_{1},\dots,\alpha_{e}$
and $\beta_{1},\dots,\beta_{f}$ of the source and the target, and
elements $c_{1},\dots,c_{e}\in S$ such that $\phi\otimes_{B[[t]]}S$
sends $\alpha_{i}$ to $c_{i}\beta_{i}$. Moreover, we may suppose
that for some $d\in\{1,\dots,e\}$, we have $c_{i}=0$, $i\le d$
and $c_{i}\ne0$, $i>d$. 

Identifications $N\cong k[[t]]^{\oplus r}$ and $\fO_{A}\cong B[[t]]^{|G|}$
induce an identification 
\[
N\otimes_{k[[t]]}\fO_{A}\otimes_{B[[t]]}S\cong S^{r|G|}.
\]
Through this identification, $\alpha_{i}$ and $\beta_{i}$ are expressed
as tuples $(\alpha_{i,j})_{j}$ and $(\beta_{i,j})_{j}$ of elements
of $S$. Note that an element of $S$ is a fraction $u/v$ with $u,v\in B[[t]]$
such that $v$ has nonzero constant term, denoted by $v_{0}$, and
$v$ is invertible in the ring $B_{v_{0}}[[t]]$. In particular there
exists $v\in B[[t]]-(t)$ such that the $S$-bases $\alpha_{i}$ and
$\beta_{j}$ are also bases over $B[[t]]_{v}$. Replacing $B$ by
$B_{v_{0}}$ we can therefore assume that this holds globally. In
particular $c_{i}\in B[[t]]$ and we may further suppose that the
leading coefficients of $c_{i}$, $i>d$ are units, that is they are
invertible. Then, for any ring map $B\longrightarrow C$, the map
$\phi_{A_{C}}=\phi_{A}\otimes_{B[[t]]}C[[t]]$ is similarly given
by $\alpha_{i}\longmapsto c_{i}\beta_{i}$, where $c_{i}\in C[[t]]$
are zero for $i\le d$ and units for $i>d$. This ends the proof.
\end{proof}
\begin{thm}
\label{thm:constructible v}The functions $v,w\colon\Delta_{G}\longrightarrow\frac{1}{|G|}\ZZ$
are locally constructible (see \ref{def:general locally constructible map}).
Similarly for the restrictions of v and $w$ to $\Delta_{G}^{\cQ}$
for a locally constructible property $\cQ$. 
\end{thm}

\begin{proof}
The second assertion is a direct consequence of the first by \ref{thm:strong P moduli of DeltaG}
and \ref{prop:locally constructible sur locally}. We are going to
prove the first assertion, that is prove that, if $Z\to\Delta_{G}$
is a map from a scheme, then the restrictions $v_{|Z}$ and $w_{|Z}$
are locally constructible. By \ref{prop:locally constructible sur locally}
and \ref{cor:stratified-uniformization} we can assume that $Z=\Spec B$
and that the associated $G$-torsor $A/B((t))$ is uniformizable.
We can further assume that $B$ is a domain and, by \ref{lem:invariants are uniformazable},
\ref{lem:torsors and disjoint union} and \ref{lem:induction and local constructibility}
we may suppose that $A=B((s))$ is also a domain and $\fO_{A}=B[[s]]$
is $G$-invariant. In particular we can now apply \ref{lem:consistent basis}
and assume the conclusion of this lemma. Let $a_{1},\dots,a_{r}$
be a $B[[t]]$-basis of $(N\otimes_{k[[t]]}\odi A)^{G}$ and write
\[
a_{j}=(a_{ij})_{i}\in(N\otimes_{k[[t]]}\odi A)\simeq\odi A^{r}.
\]
We want to use equation (\ref{eq:v(A)}). Write 
\[
d=\det(a_{ij})\in\fO_{A}=B[[s]]
\]
as $\sum_{i\ge0}d_{i}s^{i}$, $d_{i}\in B$. Then the locus where
this determinant has $s$-order $\ge l$ is the closed subset
\[
\{v\geq l/|G|\}=\{p\in\Spec B\st\ord_{A_{k(p)}}(d)\geq l\}=\bigcap_{i<l}V(d_{i})\subset\Spec B.
\]
As for the function $w$, let $\overline{a_{ij}}$ be the image of
$a_{ij}$ in $B$ and consider the matrix $(\overline{a_{ij}})\in B^{r^{2}}$.
From equation (\ref{eq:w(A)}), we need to show that the map 
\[
\Spec B\ni x\longmapsto\rank(\overline{a_{ij}})_{x}
\]
is locally constructible. The locus where this rank is less than $s$
is the zero locus of the $s\times s$ minors of $(\overline{a_{ij}})$
and is a closed subset. This completes the proof.
\end{proof}
\begin{cor}
Let $l$ be a positive integer such that $v(\Delta_{G})\subset\frac{1}{l}\ZZ$
(e.g. $l=|G|$). Integrals $\int_{\Delta_{G}}\LL^{d-v}$ and $\int_{\Delta_{G}}\LL^{w}$
are well-defined as elements of $\hat{\cM}_{k}^{\modified,l}\cup\{\infty\}$.
\end{cor}

\bibliographystyle{plain}
\bibliography{ramif-gps}

\end{document}

%% file: packages_and_functions.for_preamble.tex
\usepackage{extarrow}
\usepackage{enumerate}
\usepackage[T1]{fontenc}
\usepackage{graphicx}
\usepackage[colorlinks=true, citecolor=blue, linkcolor=blue]{hyperref}
\usepackage{latexsym}
\usepackage{mathrsfs}
\usepackage{mathtext}
\usepackage{stmaryrd}
\usepackage{tikz}
\usepackage{textcomp}
\usepackage{upgreek}

\usetikzlibrary{arrows}
\usetikzlibrary{matrix}
\usetikzlibrary{shapes}
\usetikzlibrary{snakes}
\usetikzlibrary{matrix}

\DeclareMathOperator{\Add}{\textup{Add}}
\DeclareMathOperator{\Aff}{\mathbf{Aff}}
\DeclareMathOperator{\Alg}{\textup{Alg}}
\DeclareMathOperator{\Ann}{\textup{Ann}}
\DeclareMathOperator{\Arr}{\textup{Arr}}
\DeclareMathOperator{\Art}{\textup{Art}}
\DeclareMathOperator{\Ass}{\textup{Ass}}
\DeclareMathOperator{\Aut}{\textup{Aut}}
\DeclareMathOperator{\Autsh}{\underline{\textup{Aut}}}
\DeclareMathOperator{\Bi}{\textup{B}}
\DeclareMathOperator{\CAdd}{\textup{CAdd}}
\DeclareMathOperator{\CAlg}{\textup{CAlg}}
\DeclareMathOperator{\CMon}{\textup{CMon}}
\DeclareMathOperator{\CPMon}{\textup{CPMon}}
\DeclareMathOperator{\CRings}{\textup{CRings}}
\DeclareMathOperator{\CSMon}{\textup{CSMon}}
\DeclareMathOperator{\CaCl}{\textup{CaCl}}
\DeclareMathOperator{\Cart}{\textup{Cart}}
\DeclareMathOperator{\Cl}{\textup{Cl}}
\DeclareMathOperator{\Coh}{\textup{Coh}}
\DeclareMathOperator{\Coker}{\textup{Coker}}
\DeclareMathOperator{\Cov}{\textup{Cov}}
\DeclareMathOperator{\Der}{\textup{Der}}
\DeclareMathOperator{\Div}{\textup{Div}}
\DeclareMathOperator{\End}{\textup{End}}
\DeclareMathOperator{\Endsh}{\underline{\textup{End}}}
\DeclareMathOperator{\Ext}{\textup{Ext}}
\DeclareMathOperator{\Extsh}{\underline{\textup{Ext}}}
\DeclareMathOperator{\FAdd}{\textup{FAdd}}
\DeclareMathOperator{\FCoh}{\textup{FCoh}}
\DeclareMathOperator{\FGrad}{\textup{FGrad}}
\DeclareMathOperator{\FLoc}{\textup{FLoc}}
\DeclareMathOperator{\FMod}{\textup{FMod}}
\DeclareMathOperator{\FPMon}{\textup{FPMon}}
\DeclareMathOperator{\FRep}{\textup{FRep}}
\DeclareMathOperator{\FSMon}{\textup{FSMon}}
\DeclareMathOperator{\FVect}{\textup{FVect}}
\DeclareMathOperator{\Fibr}{\textup{Fibr}}
\DeclareMathOperator{\Fix}{\textup{Fix}}
\DeclareMathOperator{\Fl}{\textup{Fl}}
\DeclareMathOperator{\Fr}{\textup{Fr}}
\DeclareMathOperator{\Funct}{\textup{Funct}}
\DeclareMathOperator{\GAlg}{\textup{GAlg}}
\DeclareMathOperator{\GExt}{\textup{GExt}}
\DeclareMathOperator{\GHom}{\textup{GHom}}
\DeclareMathOperator{\GL}{\textup{GL}}
\DeclareMathOperator{\GMod}{\textup{GMod}}
\DeclareMathOperator{\GRis}{\textup{GRis}}
\DeclareMathOperator{\GRiv}{\textup{GRiv}}
\DeclareMathOperator{\Gal}{\textup{Gal}}
\DeclareMathOperator{\Gl}{\textup{Gl}}
\DeclareMathOperator{\Grad}{\textup{Grad}}
\DeclareMathOperator{\Hilb}{\textup{Hilb}}
\DeclareMathOperator{\Hl}{\textup{H}}
\DeclareMathOperator{\Hom}{\textup{Hom}}
\DeclareMathOperator{\Homsh}{\underline{\textup{Hom}}}
\DeclareMathOperator{\ISym}{\textup{Sym}^*}
\DeclareMathOperator{\Imm}{\textup{Im}}
\DeclareMathOperator{\Irr}{\textup{Irr}}
\DeclareMathOperator{\Iso}{\textup{Iso}}
\DeclareMathOperator{\Isosh}{\underline{\textup{Iso}}}
\DeclareMathOperator{\Ker}{\textup{Ker}}
\DeclareMathOperator{\LAdd}{\textup{LAdd}}
\DeclareMathOperator{\LAlg}{\textup{LAlg}}
\DeclareMathOperator{\LMon}{\textup{LMon}}
\DeclareMathOperator{\LPMon}{\textup{LPMon}}
\DeclareMathOperator{\LRings}{\textup{LRings}}
\DeclareMathOperator{\LSMon}{\textup{LSMon}}
\DeclareMathOperator{\Left}{\textup{L}}
\DeclareMathOperator{\Lex}{\textup{Lex}}
\DeclareMathOperator{\Loc}{\textup{Loc}}
\DeclareMathOperator{\M}{\textup{M}}
\DeclareMathOperator{\ML}{\textup{ML}}
\DeclareMathOperator{\MLex}{\textup{MLex}}
\DeclareMathOperator{\Map}{\textup{Map}}
\DeclareMathOperator{\Mod}{\textup{Mod}}
\DeclareMathOperator{\Mon}{\textup{Mon}}
\DeclareMathOperator{\Ob}{\textup{Ob}}
\DeclareMathOperator{\Obj}{\textup{Obj}}
\DeclareMathOperator{\PDiv}{\textup{PDiv}}
\DeclareMathOperator{\PGL}{\textup{PGL}}
\DeclareMathOperator{\PML}{\textup{PML}}
\DeclareMathOperator{\PMLex}{\textup{PMLex}}
\DeclareMathOperator{\PMon}{\textup{PMon}}
\DeclareMathOperator{\Pic}{\textup{Pic}}
\DeclareMathOperator{\Picsh}{\underline{\textup{Pic}}}
\DeclareMathOperator{\Pro}{\textup{Pro}}
\DeclareMathOperator{\Proj}{\textup{Proj}}
\DeclareMathOperator{\QAdd}{\textup{QAdd}}
\DeclareMathOperator{\QAlg}{\textup{QAlg}}
\DeclareMathOperator{\QCoh}{\textup{QCoh}}
\DeclareMathOperator{\QMon}{\textup{QMon}}
\DeclareMathOperator{\QPMon}{\textup{QPMon}}
\DeclareMathOperator{\QRings}{\textup{QRings}}
\DeclareMathOperator{\QSMon}{\textup{QSMon}}
\DeclareMathOperator{\R}{\textup{R}}
\DeclareMathOperator{\Rep}{\textup{Rep}}
\DeclareMathOperator{\Rings}{\textup{Rings}}
\DeclareMathOperator{\Riv}{\textup{Riv}}
\DeclareMathOperator{\SFibr}{\textup{SFibr}}
\DeclareMathOperator{\SMLex}{\textup{SMLex}}
\DeclareMathOperator{\SMex}{\textup{SMex}}
\DeclareMathOperator{\SMon}{\textup{SMon}}
\DeclareMathOperator{\SchI}{\textup{SchI}}
\DeclareMathOperator{\Sh}{\textup{Sh}}
\DeclareMathOperator{\Soc}{\textup{Soc}}
\DeclareMathOperator{\Spec}{\textup{Spec}}
\DeclareMathOperator{\Specsh}{\underline{\textup{Spec}}}
\DeclareMathOperator{\Stab}{\textup{Stab}}
\DeclareMathOperator{\Supp}{\textup{Supp}}
\DeclareMathOperator{\Sym}{\textup{Sym}}
\DeclareMathOperator{\TMod}{\textup{TMod}}
\DeclareMathOperator{\Top}{\textup{Top}}
\DeclareMathOperator{\Tor}{\textup{Tor}}
\DeclareMathOperator{\Vect}{\textup{Vect}}
\DeclareMathOperator{\alt}{\textup{ht}}
\DeclareMathOperator{\car}{\textup{char}}
\DeclareMathOperator{\codim}{\textup{codim}}
\DeclareMathOperator{\degtr}{\textup{degtr}}
\DeclareMathOperator{\depth}{\textup{depth}}
\DeclareMathOperator{\divis}{\textup{div}}
\DeclareMathOperator{\et}{\textup{et}}
\DeclareMathOperator{\ffpSch}{\textup{ffpSch}}
\DeclareMathOperator{\h}{\textup{h}}
\DeclareMathOperator{\ilim}{\displaystyle{\lim_{\longrightarrow}}}
\DeclareMathOperator{\ind}{\textup{ind}}
\DeclareMathOperator{\indim}{\textup{inj dim}}
\DeclareMathOperator{\lf}{\textup{LF}}
\DeclareMathOperator{\op}{\textup{op}}
\DeclareMathOperator{\ord}{\textup{ord}}
\DeclareMathOperator{\pd}{\textup{pd}}
\DeclareMathOperator{\plim}{\displaystyle{\lim_{\longleftarrow}}}
\DeclareMathOperator{\pr}{\textup{pr}}
\DeclareMathOperator{\pt}{\textup{pt}}
\DeclareMathOperator{\rk}{\textup{rk}}
\DeclareMathOperator{\tr}{\textup{tr}}
\DeclareMathOperator{\type}{\textup{r}}
\DeclareMathOperator*{\colim}{\textup{colim}}

%% file: packages_and_functions.tex
\global\long\def\A{\mathbb{A}}%

\global\long\def\Ab{(\textup{Ab})}%

\global\long\def\C{\mathbb{C}}%

\global\long\def\Cat{(\textup{cat})}%

\global\long\def\Di#1{\textup{D}(#1)}%

\global\long\def\E{\mathcal{E}}%

\global\long\def\F{\mathbb{F}}%

\global\long\def\GCov{G\textup{-Cov}}%

\global\long\def\Gcat{(\textup{Galois cat})}%

\global\long\def\Gfsets#1{#1\textup{-fsets}}%

\global\long\def\Gm{\mathbb{G}_{m}}%

\global\long\def\GrCov#1{\textup{D}(#1)\textup{-Cov}}%

\global\long\def\Grp{(\textup{Grps})}%

\global\long\def\Gsets#1{(#1\textup{-sets})}%

\global\long\def\HCov{H\textup{-Cov}}%

\global\long\def\MCov{\textup{D}(M)\textup{-Cov}}%

\global\long\def\MHilb{M\textup{-Hilb}}%

\global\long\def\N{\mathbb{N}}%

\global\long\def\PGor{\textup{PGor}}%

\global\long\def\PGrp{(\textup{Profinite Grp})}%

\global\long\def\PP{\mathbb{P}}%

\global\long\def\Pj{\mathbb{P}}%

\global\long\def\Q{\mathbb{Q}}%

\global\long\def\RCov#1{#1\textup{-Cov}}%

\global\long\def\RR{\mathbb{R}}%

\global\long\def\WW{\textup{W}}%

\global\long\def\Z{\mathbb{Z}}%

\global\long\def\acts{\curvearrowright}%

\global\long\def\alA{\mathscr{A}}%

\global\long\def\alB{\mathscr{B}}%

\global\long\def\arr{\longrightarrow}%

\global\long\def\arrdi#1{\xlongrightarrow{#1}}%

\global\long\def\catC{\mathscr{C}}%

\global\long\def\catD{\mathscr{D}}%

\global\long\def\catF{\mathscr{F}}%

\global\long\def\catG{\mathscr{G}}%

\global\long\def\comma{,\ }%

\global\long\def\covU{\mathcal{U}}%

\global\long\def\covV{\mathcal{V}}%

\global\long\def\covW{\mathcal{W}}%

\global\long\def\duale#1{{#1}^{\vee}}%

\global\long\def\fasc#1{\widetilde{#1}}%

\global\long\def\fsets{(\textup{f-sets})}%

\global\long\def\iL{r\mathscr{L}}%

\global\long\def\id{\textup{id}}%

\global\long\def\la{\langle}%

\global\long\def\odi#1{\mathcal{O}_{#1}}%

\global\long\def\ra{\rangle}%

\global\long\def\rig{\mathbin{\!\!\pmb{\fatslash}}}%

\global\long\def\set{(\textup{Sets})}%

\global\long\def\shA{\mathcal{A}}%

\global\long\def\shB{\mathcal{B}}%

\global\long\def\shC{\mathcal{C}}%

\global\long\def\shD{\mathcal{D}}%

\global\long\def\shE{\mathcal{E}}%

\global\long\def\shF{\mathcal{F}}%

\global\long\def\shG{\mathcal{G}}%

\global\long\def\shH{\mathcal{H}}%

\global\long\def\shI{\mathcal{I}}%

\global\long\def\shJ{\mathcal{J}}%

\global\long\def\shK{\mathcal{K}}%

\global\long\def\shL{\mathcal{L}}%

\global\long\def\shM{\mathcal{M}}%

\global\long\def\shN{\mathcal{N}}%

\global\long\def\shO{\mathcal{O}}%

\global\long\def\shP{\mathcal{P}}%

\global\long\def\shQ{\mathcal{Q}}%

\global\long\def\shR{\mathcal{R}}%

\global\long\def\shS{\mathcal{S}}%

\global\long\def\shT{\mathcal{T}}%

\global\long\def\shU{\mathcal{U}}%

\global\long\def\shV{\mathcal{V}}%

\global\long\def\shW{\mathcal{W}}%

\global\long\def\shX{\mathcal{X}}%

\global\long\def\shY{\mathcal{Y}}%

\global\long\def\shZ{\mathcal{Z}}%

\global\long\def\st{\ | \ }%

\global\long\def\stA{\mathcal{A}}%

\global\long\def\stB{\mathcal{B}}%

\global\long\def\stC{\mathcal{C}}%

\global\long\def\stD{\mathcal{D}}%

\global\long\def\stE{\mathcal{E}}%

\global\long\def\stF{\mathcal{F}}%

\global\long\def\stG{\mathcal{G}}%

\global\long\def\stH{\mathcal{H}}%

\global\long\def\stI{\mathcal{I}}%

\global\long\def\stJ{\mathcal{J}}%

\global\long\def\stK{\mathcal{K}}%

\global\long\def\stL{\mathcal{L}}%

\global\long\def\stM{\mathcal{M}}%

\global\long\def\stN{\mathcal{N}}%

\global\long\def\stO{\mathcal{O}}%

\global\long\def\stP{\mathcal{P}}%

\global\long\def\stQ{\mathcal{Q}}%

\global\long\def\stR{\mathcal{R}}%

\global\long\def\stS{\mathcal{S}}%

\global\long\def\stT{\mathcal{T}}%

\global\long\def\stU{\mathcal{U}}%

\global\long\def\stV{\mathcal{V}}%

\global\long\def\stW{\mathcal{W}}%

\global\long\def\stX{\mathcal{X}}%

\global\long\def\stY{\mathcal{Y}}%

\global\long\def\stZ{\mathcal{Z}}%

\global\long\def\then{\ \Longrightarrow\ }%

\global\long\def\L{\textup{L}}%

\global\long\def\l{\textup{l}}%